\documentclass{amsart}
\usepackage{amsmath,amsthm,mathtools, amsfonts,amssymb}
\usepackage[utf8]{inputenc}
\usepackage[dvipsnames]{xcolor}
\usepackage{graphicx}
\usepackage{graphicx} 
\usepackage{color}
\usepackage{tikz}
\usepackage{tikz-cd}
\usepackage[all]{xy}

\usepackage{extarrows}
\usepackage{ytableau}

\usepackage[mathscr]{euscript}

\allowdisplaybreaks[4]
\numberwithin{equation}{section}
\usepackage[colorinlistoftodos]{todonotes}

\newcommand{\G}{\mathfrak{G}}

\newcommand{\Z}{\mathbb{Z}}
\newcommand{\SP}{\mathscr{SP}}
\newcommand{\hL}{\hat{\Lambda}}
\newcommand{\Rep}{\mathscr{R}}
\newcommand{\Winf}{W_\infty}
\newcommand{\Winfgr}{\Winf^0}
\newcommand{\RR}{\mathscr{R}}
\newcommand{\lvac }{\langle \mathrm{vac} |}
\newcommand{\rvac}{|\mathrm{vac} \rangle}

\newtheorem{thm}{Theorem}[section]
\newtheorem{lem}[thm]{Lemma}
\newtheorem{prop}[thm]{Proposition}
\newtheorem{cor}[thm]{Corollary}
\newtheorem{conj}[thm]{Conjecture}

\theoremstyle{remark}
\newtheorem{example}[thm]{Example}

\newtheorem{remark}[thm]{Remark}

\begin{document}

\title{Neutral-Fermion constructions of factorial $gp$-and $gq$-Functions}

\author[Koushik Brahma, Takeshi Ikeda, Shinsuke Iwao, Yi Yang]{Koushik Brahma$^1$, Takeshi Ikeda$^1$, Shinsuke Iwao$^2$, Yi Yang$^{3}$
}
\dedicatory{
$^1$ Waseda University, Faculty of Science and Engineering\\
3-4-1 Okubo, Shinjuku-ku
Tokyo 169-8555, Japan\\
$^2$ Keio University, Faculty of Business and Commerce\\
4-1-1 Hiyoshi, Kohoku-ku
Yokohama, Kanagawa 223-8521, Japan \\
$^3$Sun Yat-sen University, School of Mathematics\\
135 Xin'gang West Road, Haizhu,
Guangzhou 510275, China \\
Email addresses: koushikbrahma95@gmail.com, gakuikeda@waseda.jp,\\ iwao-s@fbc.keio.ac.jp, yangy875@mail2.sysu.edu.cn
}

\date{\today}


\begin{abstract}
We develop neutral-fermionic constructions for the factorial \(gp\)- and \(gq\)-functions introduced by Nakagawa and Naruse, which are respectively dual to the factorial \(GQ\)- and \(GP\)-functions of Ikeda and Naruse. In particular, we realize the factorial $GP$-, $GQ$- and \(gq\)-functions as vacuum expectation values. As applications, we obtain, Jacobi--Trudi type determinantal formulas for the transition coefficients between functions with different equivariant parameters for \(gq\) and its dual \(GP\), as well as a Pfaffian formula for the factorial \(gq\)-functions. We further prove a remarkable coincidence among the transition coefficients for parameter changes for \(gp\), \(gq\), \(GQ\), and \(GP\). These coefficients admit a description  in terms of factorial Grothendieck polynomials of type A.
\end{abstract}

\maketitle

\tableofcontents

\section{Introduction}
Symmetric functions originating from \( K \)-theoretic Schubert calculus provide a rich algebraic framework connecting geometry, combinatorics, and representation theory.
Among these, the factorial \( gp \)-functions, introduced by Nakagawa and Naruse \cite{NN:Uni}, form a family of inhomogeneous symmetric functions in $y=(y_1,y_2,\ldots)$ depending on an infinite sequence of parameters \( b = (b_1, b_2, \ldots) \).
Each \( gp_\lambda(y\,|\,b) \), indexed by a strict partition \( \lambda \), is defined as the dual of the \( K \)-theoretic factorial Schur \( Q \)-function \( GQ_\lambda(x\,|\,b) \) introduced by Ikeda and Naruse \cite{IN13}, which represents the Schubert basis in the torus-equivariant \( K \)-theory of the Lagrangian Grassmannian.


An affine counterpart of these functions was recently introduced by Ikeda, Shimozono, and Yamaguchi~\cite{ISY2}
in connection with the affine Grassmannian of the symplectic group $\mathrm{Sp}_{2n}(\mathbb{C})$.
The resulting \emph{affine $gp$-functions}, indexed by a certain set of partitions $\lambda$, form the ideal-sheaf basis of the corresponding equivariant $K$-homology ring. The $gp$ function can be regarded as the limit, as
$n\to \infty$, of affine $gp$ functions. More precisely,
when $|\lambda| \le 2n$, they coincide with $gp_{\lambda}(y | b)$ under a suitable specialization of the equivariant parameters $b_i$.
This affine version has motivated further study of the underlying combinatorial and algebraic structures of \( gp \)-functions.

In this paper we also study the factorial \( gq \)-functions, defined as the duals to the factorial \( GP \)-functions that represent the Schubert basis in the \( K \)-theory of the maximal orthogonal Grassmannian.
Our aim is to provide a unified framework for these families of functions through the free–Fermion formalism, which generalizes the non-equivariant results of Iwao~\cite{Iwao2023-2}.

We express the functions
\( GP_\lambda(x\,|\,b) \),
\( GQ_\lambda(x\,|\,b) \), and
\( gq_\lambda(y\,|\,b) \)
as vacuum expectation values (Theorem~\ref{thm:GP as VEV}, Theorem~\ref{thm: gq VEV}).
We provide a similar expression
for $gp_\lambda(y\,|\,b)$ as a conjecture (Conjecture \ref{conj:gp}).
As a corollary, we obtain a Schur-type Pfaffian formula for \( gq_\lambda(y\,|\,b) \) (Corollary \ref{cor:Pf gq}).
Note that such a formula has previously been known for \( GQ_\lambda(x\,|\,b) \) and \( GP_\lambda(x\,|\,b) \) due to the work of Hudson, Ikeda, Matsumura, and Naruse.


Based on the free-Fermion formalism,
we further study
the explicit expansion formula of
these functions. We define
\begin{align*}
GP_\lambda(x\,|\,b)&=\sum_{\mu}
C^{GP}_{\lambda\mu}(b,c)
GP_{\mu}(x\,|\,c),&
GQ_\lambda(x\,|\,b)&=\sum_{\mu}
C^{GQ}_{\lambda\mu}(b,c)
GQ_{\mu}(x\,|\,c),\\
gp_\lambda(x\,|\,b)&=\sum_{\mu}
C^{gp}_{\lambda\mu}(b,c)
gp_{\mu}(x,c),&
gq_\lambda(x\,|\,b)&=\sum_{\mu}
C^{gq}_{\lambda\mu}(b,c)
gq_{\mu}(x\,|\,c).
    \end{align*}
We obtain a Jacobi–Trudi type expression for these coefficients (Theorem \ref{thm:gq(xb)}), together with an alternative formula in terms of type~A Grothendieck polynomials (Theorem \ref{thm:gp expansion}).
As a byproduct of the argument, we obtained the following unexpected coincidence of coefficients:
    \begin{equation}
C_{\lambda\mu}^{gq}(b,c)
=C_{\lambda\mu}^{gp}(b,c)=
C_{\mu\lambda}^{GP}(c,b)
=C_{\mu\lambda}^{GQ}(c,b).\label{eq:coincidence}    \end{equation}

    The paper organized as follows. In \S~ \ref{sec_GP,GQ}, we fix notation for root systems, Weyl groups and strict partitions. We explain the definition and some properties of the factorial $GP$- and $GQ$-functions. We introduce the factorial $gp$- and $gq$-functions as the dual to $GQ$ and $GP$ respectively. We also provide operator constructions for these functions.
    In \S~ \ref{sec_Free-Fermion formalism}, we introduce the neutral-free Fermion, and some of its deformations. The goal of this section is to express $GP_\lambda(x\,|\,b)$ and $gq_\lambda(y\,|\,b)$ as vacuum expectation values for the Fock space of neutral-free-Fermion. As a corollary, we obtain a Pfaffian formula for $gq_\lambda(y\,|\,b).$
    In \S~ \ref{sec:JT}, we investigate the coefficients $C_{\lambda\mu}^{gq}(b,c)$ and $C_{\lambda\mu}^{GP}(b,c)$ within the
    framework of Fermion formalism.
    We obtain a Jacobi-Trudi type formula for these coefficients. In \S~ \ref{sec_coef_gro_pol},
    we prove another expressions for the coefficients
    $C_{\lambda\mu}^{gp}(b,c)$ and $C_{\lambda\mu}^{GQ}(b,c)$ in terms of type A double Grothendieck polynomials. As a byproduct of this
    computation, we prove the coincidence of coefficients \eqref{eq:coincidence} and  give a Jacobi-Trudi type formulae of factorial Grothendieck polynomials.

\subsection*{Acknowledgements}
This work was supported by JSPS KAKENHI Grant Numbers
24KF0258, 25KF0074, 
23K25772, 22K03239, 
23K03056. 
We are grateful to Professor Hiroshi Naruse for kindly sharing the slides from his talk in the workshop on Algebraic and Enumerative Combinatorics
at Shinshu University, Japan.
The first author was supported by a JSPS Postdoctoral Fellowship for Research in Japan, during which this research was carried out.
The main part of this work was carried out while the fourth author was visiting Waseda University. The authors would like to thank Waseda University for its hospitality.

\section{Factorial $GP,GQ$ and $gp,gq$-functions}\label{sec_GP,GQ}
In this section, we review some results on factorial $GP$, $GQ$ functions following \cite{IN13} as well as their duals factorial $gq$, $gp$ functions following \cite{NN23}.

\subsection{Preliminaries on root systems and Weyl groups, and strict partitions}

\subsubsection{Root systems and Weyl groups}
Let $\mathrm{B}_\infty$ and $\mathrm{C}_\infty$ be the following Dynkin diagrams of infinite rank respectively:
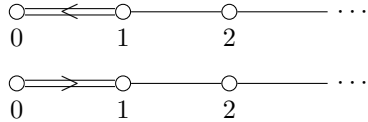
\begin{center}
\begin{tikzpicture}[node distance=1.5cm and 1.2cm, auto]
    \node[draw, circle, inner sep=2pt, label=below:{$0$}] (0) {};
    \node[draw, circle, inner sep=2pt, label=below:{$1$}, right=of 0] (1) {};
    \node[draw, circle, inner sep=2pt, label=below:{$2$}, right=of 1] (2) {};
    \node[right=of 2] (dots) {\(\cdots\)};

    \draw[double distance=2pt] (0) -- (1) node[midway, yshift=-1.5ex] {\(<\)};

    \draw (1) -- (2);
    \draw (2) -- (dots);
 \end{tikzpicture}
 \end{center}
\begin{center}
\begin{tikzpicture}[node distance=1.5cm and 1.2cm, auto]
    \node[draw, circle, inner sep=2pt, label=below:{$0$}] (0) {};
    \node[draw, circle, inner sep=2pt, label=below:{$1$}, right=of 0] (1) {};
    \node[draw, circle, inner sep=2pt, label=below:{$2$}, right=of 1] (2) {};
    \node[right=of 2] (dots) {\(\cdots\)};

    \draw[double distance=2pt] (0) -- (1) node[midway, yshift=-1.5ex] {\(>\)};

    \draw (1) -- (2);
    \draw (2) -- (dots);
\end{tikzpicture}
\end{center}
For both types,
the set of vertex is
$I=\{0,1,2,\ldots\}$.
The Weyl group is the infinite hyper-octhedral
group $W_\infty$ generated by
$s_i\;(i\in I)$ satisfying the relations
$$
s_i^2=1\;(i\in I),
\quad
s_0s_1s_0s_1=s_1s_0s_1s_0,\quad
s_is_{i+1}s_i=s_{i+1}s_is_{i+1}\quad \text{for $i\ge 1$.}
$$
The group $W_\infty$ acts on the free abelian group $L:=\bigoplus_{i=1}^\infty \Z a_i$
by
\begin{equation*}
s_0(a_j)=\begin{cases}
-a_1 & \text{if $j=1$,}\\
a_j & \text{otherwise,}
\end{cases},\quad
s_i(a_j) =\begin{cases}
    a_{i+1}&\text{if $j=i$,}\\
    a_i & \text{if $j=i+1$,}\\
    a_j & \text{otherwise.}
\end{cases}
\end{equation*}
The simple roots $\{\alpha_i\}_{i\in I}$ of type $\mathrm{B}_\infty$ are defined by
$$
\alpha_0=-a_1,\quad
\alpha_i=a_i-a_{i+1}\quad(i\ge 1),
$$
while the simple roots $\{\alpha_i\}_{i\in I}$  of type $\mathrm{C}_\infty$ are defined by
$$
\alpha_0=-2a_1,\quad
\alpha_i=a_i-a_{i+1}\quad(i\ge 1).
$$
After this point, results for types $\mathrm{B}$ and $\mathrm{C}$ will often be presented in parallel.
Throughout this paper, we deliberately use the same notation for both types without changing symbols.
The interpretation should be understood from the context.

\subsubsection{Strict Partitions and Shifted Diagrams}

A strict partition $\lambda$ is a strictly decreasing sequence of positive integers
\[
\lambda = (\lambda_1> \lambda_2>\dots >\lambda_\ell>0).\]
Let $\SP$ denote the set of all strict partitions.
We identify $\lambda = (\lambda_1 > \cdots > \lambda_\ell) \in \SP$ with its shifted Young diagram consisting of the boxes
\[
(i,j) \quad \text{with} \quad 1 \le i \le \ell,\;\; i \le j \le i+\lambda_i-1,
\]
drawn in matrix-style coordinates (row $i$, column $j$). The size of $\lambda$ is defined by
$|\lambda| = \sum_{i=1}^\ell \lambda_i,
$
which is equal to the total number of boxes in the corresponding shifted Young diagram.
The length of $\lambda$ is $\ell$, which is the number of rows of the diagram.
For each box $(i,j)$ in the shifted Young diagram, we define its \emph{content} by
\[
\mathfrak{c}(i,j) = j - i.
\]
Thus, boxes lying on the same diagonal have the same content.

\begin{example}\label{ex_st_par}
  The shifted diagram of $\lambda = (5, 3, 1)\in \SP$ is
 depicted below, with each box labeled by its contents:
\[
\begin{ytableau}
0 & 1 & 2 & 3&4 \\
\none & 0 & 1&2 \\
\none & \none & 0
\end{ytableau}
\]

\end{example}

Fix an integer $i\ge 0$.
An \emph{$i$-addable} box of $\lambda\in \SP$ is a box of content $i$ that can be added to the shifted Young diagram of $\lambda$ so that the result is again a shifted Young diagram of a strict partition.
Similarly, an \emph{$i$-removable} box of $\lambda\in \SP$ is a box of content $i$ that can be removed from the diagram of $\lambda$ so that the remaining diagram is still a shifted Young diagram of a strict partition.
If $\lambda$ has a $i$-addable box, define $s_i \lambda$ to be the strict partition obtained by adding that box.
If $\lambda$ has a $i$-removable box, define $s_i \lambda$ to be the strict partition obtained by removing that box.
Otherwise, set $s_i \lambda = \lambda$.
This defines an action of $\Winf$ on $\SP.$

Let $S_\infty$ be the subgroup of $\Winf$ generated by the simple reflections
$s_i \; (i \ge 1).
$
Let $\Winfgr$ denote the set of minimal-length coset representatives of $\Winf / S_\infty$, namely
\[
\Winfgr
=
\{\, w \in \Winf \mid \ell(ws_i)=\ell(w)+1 \text{ for all } i \ge 1 \,\}.
\]

There is an isomorphism between $\Winfgr$ and $\SP$.
For $\lambda=(\lambda_1,\dots,\lambda_r)\in \SP$ and $i \ge 1$, define
$
\rho_i := s_{i-1}\cdots s_1 s_0.$
Then set
\begin{equation}\label{eq_w_lam_A}
w_\lambda
:=
\rho_{\lambda_r}\cdots \rho_{\lambda_1}.
\end{equation}
The element $w_\lambda$ is the minimal-length coset representative in $\Winfgr$
corresponding to $\lambda \in \SP$.

\subsection{Factorial $GP$- and $GQ$-functions}
This subsection reviews the definition and basic properties of the factorial
$GP$- and $GQ$-functions.
\subsubsection{Notation in the $K$-theoretic Setting}

Let $\beta$ be an indeterminate. We use the following
notation
$$
x\oplus y=x+y+\beta xy,\quad
x\ominus y=\frac{x-y}{1+\beta y},\quad
\ominus x= 0\ominus x=\frac{-x}{1+\beta x}.
$$
Let $\mathbb{F}=\mathbb{Q}(\beta)(b_1,b_2,\ldots)$ be the field of  rational functions generated by $b_1,b_2,\ldots$ and $\beta$.
Let $\RR$ be the $\Z[\beta]$- subalgebra of
$\mathbb{F}$ generated by
$b_i$ and $\ominus{b}_i$ for $i=1,2,\ldots$.

Define, for type $\mathrm{B}_\infty$
\begin{equation}\label{eq_we_B_infty}
c(\alpha_0)=b_1,\quad c(\alpha_i)=b_{i+1}\ominus b_i\quad \text{for $i\ge 1$.}
\end{equation}
and, for type $\mathrm{C}_\infty$
\begin{equation}\label{eq_we_C_infty}
c(\alpha_0)=b_1\oplus b_1,\quad c(\alpha_i)=b_{i+1}\ominus b_i\quad \text{for $i\ge 1$.}
\end{equation}
The ring $\RR$ is a deformation of the group ring $\Z[L]=\Z[e^{\pm a_i}\;(i\ge 1)]$ of $L$.
In fact, via the specialization map $\Z[\beta]\rightarrow \Z$ given by
$\beta\mapsto -1$, $\Z\otimes_{\Z[\beta]}\RR$ is
isomorphic to the group ring $\Z[L]$ of $L$.
More precisely, we identify the generators by
$$
b_i\mapsto  1-e^{-a_i},\quad
\ominus b_i\mapsto 1-e^{a_i}.
$$
Then for each $i\in I$, the element
$c(\alpha_i)\in \RR$ specializes to $1-e^{\alpha_i}\in \Z[L].$

For $k\ge 1$, define
\begin{align}
[\![x|b]\!]^k&=(x\oplus x)(x\ominus b_1)\cdots (x\ominus b_{k-1}),\label{eq:factorial for GQ}\\
[x|b]^k&=(x\ominus b_1)\cdots (x\ominus b_{k}).
\end{align}
The convention for equivariant parameters adopted in this paper differs from that in \cite{IN13} only in that the parameter
$b_i$ in \cite{IN13} corresponds to
$\ominus b_i$ in the present work.

\subsubsection{Definition of $GP$ and $GQ$}
 Let $\lambda\in \SP.$ Let $r=\ell(\lambda).$
 For $n\ge r$,
 we define
\begin{align}
    GQ_\lambda(x_1,\ldots,x_n|b)&=
    \frac{1}{(n-r)!}
    \sum_{w\in S_n}
    w\left(
    \prod_{i=1}^{r}[\![x_i|b]\!]^{\lambda_i}
    \prod_{i=1}^r \prod_{j=i+1}^n
    \frac{x_i\oplus x_j}{x_i\ominus x_j}
    \right),\\
GP_\lambda(x_1,\ldots,x_n|b)&=
    \frac{1}{(n-r)!}
    \sum_{w\in S_n}
    w\left(
\prod_{i=1}^rx_i[x_i|b]^{\lambda_i-1}
    \prod_{i=1}^r \prod_{j=i+1}^n
    \frac{x_i\oplus x_j}{x_i\ominus x_j}
    \right),
    \label{eq:GP}
\end{align}
where $w$ permutes $x_1,\ldots,x_n$. These are symmetric polynomials in $x=(x_1,\ldots,x_n)$ with coefficients in $\Rep.$ We have the stability property
\begin{align*}
GQ_\lambda(x_1,\ldots,x_{n},0|b)&=GQ_\lambda(x_1,\ldots,x_{n}|b),\\
GP_\lambda(x_1,\ldots,x_{n},0|b)&=GP_\lambda(x_1,\ldots,x_{n}|b),
\label{eq:GP}
\end{align*}
and so the infinite variable version
$GQ_\lambda(x_1,x_2,\ldots|b)$ and  $GP_\lambda(x_1,x_2,\ldots|b)$
makes sense, which we simply denote by $GQ_\lambda(x\,|\,b)$
and $GP_{\lambda}(x\,|\,b)$ respectively.

By setting all parameters $b_i$ to zero, we obtain the
non-equivariant versions
\begin{align*}
GQ_\lambda(x):=GQ_\lambda(x|0),\quad
GP_\lambda(x):=GP_\lambda(x|0).
\end{align*}
Define
\begin{equation}
G\Gamma_{\mathscr{R}}(x)
  := \bigoplus_{\lambda\in \SP} \mathscr{R}\, GP_\lambda(x),
\qquad
G\Gamma_{\mathscr{R}}^+(x)
  := \bigoplus_{\lambda\in \SP} \mathscr{R}\, GQ_\lambda(x).
\end{equation}
Then
$\{GP_\lambda(x| b)\,|\, \lambda\in \SP\}$
(resp. $\{GQ_\lambda(x| b)\,|\, \lambda\in \SP\}$)
forms an $\mathscr{R}$-basis of
$G\Gamma_{\mathscr{R}}(x)$
(resp. $G\Gamma_{\mathscr{R}}^+(x)$).
It is known that both are $\mathscr{R}$-algebras, and that
$G\Gamma_{\mathscr{R}}^+(x)$ is a subalgebra of
$G\Gamma_{\mathscr{R}}(x)$; see \cite{IN13} and \cite{CM} for details.
\begin{remark}
The family of functions $GP_\lambda(x\,|\,b)$ defined here is of type B, in the sense that they correspond to the Schubert basis of the torus equivariant $K$-theory for the maximal isotropic
Grassmannian of type B (see \cite[\S~ 8]{IN13}). More precisely,
the function defined in \eqref{eq:GP} is
$GB_\lambda^{(n)}(x|\ominus b)=GP_{\lambda}(x_1,\ldots,x_n|0,\ominus b)$ in the notation of \cite[\S~ 6.1]{IN13}.
\end{remark}

\subsubsection{Vanishing property}
 For each strict partition $\mu=(\mu_1,\cdots,\mu_r)$ of length $r$ define a sequence $b_{\mu}$ by
 $$b_\mu=(b_{\mu_1},\cdots,b_{\mu_r},0,\cdots )$$
\begin{prop}[Vanishing property, \cite
{IN13}]\label{prop_van_pro}
 Let $\lambda$ and $\mu$ be strict partitions. If $\lambda\not\subseteq\mu$ then $GQ_\lambda(b_\mu|b)=0$, and $GP_\lambda(b_\mu|b)=0$. Moreover, we have $GQ_\lambda(b_\lambda|b)\neq 0$, $GP_\lambda(b_\lambda|b)\neq 0$.
\end{prop}

\subsubsection{Factorization property}
For a partition $\lambda$ of length less than or equal to $r$, define the \emph{factorial Grothendieck polynomial} by
\begin{equation}\label{eq: def fac Groth}
G_\lambda(x_1,\ldots,x_r|b)
=\sum_{w\in S_r}
w\left(
\frac{\prod_{i=1}^{r}[x_i|b]^{\lambda_i+r-i}}{\prod_{1\le i<j\le r}(x_i\ominus x_j)}
\right).
\end{equation}
The following factorization formula holds
which will be used in \S~ \ref{sec_coef_gro_pol}.
\begin{lem}[Factorization formula for factorial $GP$- and $GQ$-functions]\label{lem_fac_GP}
For a partition $\lambda$ of length less than or equal to $r$, we have
\begin{align}
GQ_{\delta_r+\lambda}(x_1,\ldots,x_r|b)
&=GQ_{\delta_r}(x_1,\ldots,x_r|b)
G_\lambda(x_1,\ldots,x_r|b),\label{eq: GQ rho+la 1}\\
GP_{\delta_r+\lambda}(x_1,\ldots,x_r|b)
&=GP_{\delta_r}(x_1,\ldots,x_r|b)
G_\lambda(x_1,\ldots,x_r|b),\label{eq: GP rho+la 1}\\
GP_{\delta_r}(x_1,\ldots,x_r|b)&=
GP_{\delta_r}(x_1,\ldots,x_r)
=x_1\cdots x_r\prod_{1\le i<j\le r}(x_i\oplus x_j),\label{eq:GP rho 0}\\
GQ_{\delta_r}(x_1,\ldots,x_r|b)&=
GQ_{\delta_r}(x_1,\ldots,x_r)
=\prod_{1\le i\leq j\le r}(x_i\oplus x_j).\label{eq:GQ rho 0}
\end{align}
\end{lem}
\begin{proof}
The factorization in \eqref{eq: GQ rho+la 1} was proved in \cite[Proposition 3.2]{IN13}. We write a proof for \eqref{eq: GP rho+la 1}.
Since $\ell(\delta_r+\lambda)=r$, by using \eqref{eq:GP} we have,
\begin{align}
GP_{\delta_r+\lambda}(x_1,\ldots,x_r|b)&=\sum_{w\in S_r}
w\left(
\prod_{i=1}^rx_i
[x_i|b]^{\lambda_i+r-i}
\prod_{1\le i<j\le r}\frac{x_i\oplus x_j}{x_i\ominus x_j}
\right)\nonumber\\
&=
x_1\cdots x_r\prod_{1\le i<j\le r}(x_i\oplus x_j)
\sum_{w\in S_r}
w\left(
\frac{\prod_{i=1}^r[x_i|b]^{\lambda_i+r-i}
}{\prod_{1\le i<j\le r}(x_i\ominus x_j)}
\right)\nonumber\\
&=
x_1\cdots x_r\prod_{1\le i<j\le r}(x_i\oplus x_j)
\cdot G_\lambda(x_1,\ldots,x_r|b),\label{eq:GP rho+la}
\end{align}
where we used
the fact that $x_1\cdots x_r\prod_{1\le i<j\le r}(x_i\oplus x_j)$ is symmetric. In particular,  for  $\lambda=\varnothing,$ we have
\begin{equation}
GP_{\delta_r}(x_1,\ldots,x_r|b)=x_1\cdots x_r\prod_{1\le i<j\le r}(x_i\oplus x_j)
\label{eq:GP rho}
\end{equation}
because $G_\varnothing(x_1,\ldots,x_r|b)=1.$ By setting $b=0$ in \eqref{eq:GP rho}, we obtain \eqref{eq:GP rho 0}, and
hence \eqref{eq: GP rho+la 1} from  \eqref{eq:GP rho+la}.
\end{proof}

\subsubsection{Demazure operators}\label{sec:Demazure}

The group $\Winf$ acts on $G\Gamma_\Rep(x)$ by
\begin{align*}
s_0 f(x|\, b)
  &= f(\ominus b_1, x_1, x_2, \ldots \mid\, \ominus b_1, b_2, b_3, \ldots), \\
s_i f(x|\, b)
  &= f(x \mid \ldots, b_{i+1}, b_i, \ldots)
  \quad \text{for } i \ge 1.
\end{align*}
In particular, the definition of $s_0$ relies on the
$K$-theoretic $Q$-cancellation property of $G\Gamma_\Rep(x)$ defined in \cite{IN13};
see \cite[\S~ 2]{IN13} for more details.
Moreover,
$G\Gamma_\Rep^+(x)$ is
stable under the $W_\infty$ action.

Define
\begin{equation}
D_i f(x\,|\,b)=\frac{1}{c(-\alpha_i)}\left(s_i f(x\,|\,b)-(1+\beta c(-\alpha_i)f(x\,|\,b)\right).\label{eq:Dix}
\end{equation}
Because our convention for $b_i$ differs from that in \cite{IN13},
the definition of $D_i$ (called $\pi_i$ in \cite{IN13})
is adjusted accordingly.
Then the following relations holds:
\begin{equation}
    D_i^2=-\beta D_i
\end{equation}
\begin{align}
D_0D_1D_0D_1=D_1D_0D_1D_0,\quad D_iD_{i+1}D_i=D_{i+1}D_iD_{i+1} \quad \text{for $i\ge 1$.}
\end{align}
As above, we use the same notation $D_i$ for types $\mathrm{B}$ and $\mathrm{C}$. In the definition of $D_i$, the only difference is that
$c(\alpha_i)$ for $i=0$ differs between the two types.
\begin{thm}[\cite{IN13}]\label{thm:DGQ}
Let $X$ be either $P$ or $Q$, and write $GX_\lambda$ to denote
$GP_\lambda$ as well as $GQ_\lambda$. Then, for $i\in I$ we have
\begin{align*}
D_iGX_{\lambda}(x\,|\,b)
&=\begin{cases}
GX_{s_i \lambda}(x\,|\,b) & \text{if $s_i w_\lambda < w_\lambda$,}\\
-\beta GX_\lambda(x\,|\,b) & \text{if $s_i w_\lambda > w_\lambda$,}
\end{cases}
\end{align*}
\end{thm}

\subsection{Factorial $gp$- and $gq$-functions}
Let $\hat{\Lambda}(y)$ be the completed ring of
symmetric functions in $y=(y_1,y_2,\ldots)$;
an element of $\hat{\Lambda}(y)$ is a formal sum
$\sum_{n=0}^\infty f_n$, where $f_n$ is a symmetric function of degree $n$. Define
$\hat{\Lambda}_\Rep(y):=\Rep\otimes_{\Z}\hat{\Lambda}(y).$

There exist unique families
$\{gp_\lambda(y\,|\,b)\}_{\lambda\in \SP}$ and
$\{gq_\lambda(y\,|\,b)\}_{\lambda\in \SP}$
in $\hat{\Lambda}_\Rep(y)$
such that
\begin{align}
\label{eq:GQ-gp}
\prod_{i,j=1}^\infty \frac{1-\bar{x}_i y_j}{1-x_i y_j}
&=\sum_{\lambda\in \SP} GQ_\lambda(x\,|\,b)\, gp_\lambda(y\,|\,b),
\\
\label{eq:GP-gq}
\prod_{i,j=1}^\infty \frac{1-\bar{x}_i y_j}{1-x_i y_j}
&=\sum_{\lambda\in \SP} GP_\lambda(x\,|\,b)\, gq_\lambda(y\,|\,b).
\end{align}


Define the following element in $\hL_\Rep(y)$
\begin{equation}
\Omega(b_i|y):=
\prod_{j=1}^\infty
\frac{1-\overline{b}_iy_j}{1-b_i y_j}.
\end{equation}

The group $\Winf$ acts on $\hL_\Rep{(y)}$ by
\begin{align*}
s_0 f(x|\, b)
  &= \Omega(b_1|y)f(x\mid \ominus b_1, b_2, b_3, \ldots), \\
s_i f(x|\, b)
  &= f(x \mid \ldots, b_{i+1}, b_i, \ldots)
  \quad \text{for } i \ge 1.
\end{align*}
For $i \in I$, the operator $T_i$ on $\hat{\Lambda}_\Rep(y)$ is defined by
\begin{equation}\label{eq_T_i}
T_i f(y\,|\,b)
  = c(\alpha_i)^{-1}\bigl(s_i f(y\,|\,b) - f(y\,|\,b)\bigr),
\end{equation}
where $c(\alpha_i)$ is defined by \eqref{eq_we_C_infty} in type $\mathrm{C}_\infty$
and by \eqref{eq_we_B_infty} in type $\mathrm{B}_\infty$.
These operators satisfy
\begin{equation}
    T_i^2=\beta T_i\quad \text{for $i \in  I$},\quad and
\end{equation}
\begin{align}
  T_0T_1T_0T_1=T_1T_0T_1T_0,\quad T_iT_{i+1}T_i=T_{i+1}T_iT_{i+1} \quad \text{for $i\ge 1$}  .
\end{align}


\begin{thm}[\cite{NN:Uni}]\label{thm_ope_des}
Let $x$ be either $p$ or $q$, and write $gx_\lambda$ for
$gp_\lambda$ or $gq_\lambda$ accordingly. Then for $i\in I$ we have
\[
T_i gx_\lambda(y\,|\,b)=
\begin{cases}
gx_{s_i \lambda}(y\,|\,b)
& \text{if $s_i w_\lambda>w_\lambda$ and $s_i w_\lambda\in \Winfgr$,}\\
0
& \text{if $s_i w_\lambda>w_\lambda$ and $s_i w_\lambda\notin \Winfgr$,}\\
\beta \, gx_{\lambda}(y\,|\,b)
& \text{if $s_i w_\lambda<w_\lambda$.}
\end{cases}
\]
Moreover, for $\lambda\in\SP$, we have
\[
gx_\lambda(y\,|\,b)=T_{w_\lambda}(1).
\]
\end{thm}
\begin{proof}
The result is due to \cite[Prop.~5.1]{NN:Uni} (see \cite[Appendix A]{ISY2} for a proof).
\end{proof}

   \section{Free-Fermion formalism}\label{sec_Free-Fermion formalism}
   In this section, after recalling neutral fermions and the Fock space, we review the results of Iwao \cite{Iwao2023-2}. In \S~\ref{sec:FacGPGQ}, we express the factorial \(GP\)- and \(GQ\)-functions as vacuum expectation values.
In \S~\ref{se:fac-gq}, we give a vacuum expectation value formula (Theorem \ref{thm: gq VEV}) for the factorial \(gq\)-functions. As consequences, we obtain their generating function formula (Corollary \ref{cor:gen-fac-gq})  and Pfaffian formula in \S~ \ref{sec:Pf_gq}. In \S~ \ref{cor:gen-fac-gq}, we present a general expansion theorem (Theorem \ref{thm:gq(xb)}), whose proof is given in \S~ \ref{Lem:phi-Phi}.
In \S~ \ref{sec:conj}, we state conjectures on the factorial $gp$-functions.

\subsection{Neutral Fermions and Fock spaces}
In this section, we give a brief summery of neutral fermions, Fock module, and vacuum expectation value.

Let $\mathcal{C}$ be the accociative $\mathbb{Q}(\beta)-$algebra
defined by the generators $\{\phi_{n}\}_{n\in \mathbb{Z}}$ and the relations:
\begin{align}
[\phi_m,\phi_n]_{+}=2(-1)^{m}\delta_{m+n,0},\label{phi-relation}
\end{align} where $[A,B]_{+}:=AB+BA$ is the \emph{anti-commutator}.
Let $\mathcal{F}$ denote the left $\mathcal{C}$--module
generated by the vector $\rvac$
such that \begin{equation}\phi_{n}\rvac=0
\quad
\text{for $n<0.$}
\label{eq:rvac}
\end{equation}
The vector $\rvac$ is called the \emph{vacuum vector}.
Let $\mathcal{F}^\dagger$ denote the right $\mathcal{C}$--module
generated by the vector $\langle 0|$, the \emph{dual vacuum vector},
such that \begin{equation}
    \lvac\phi_n=0
    \quad\text{for $n>0.$}
    \label{eq:lvac}
\end{equation}
$\mathcal{F}$ (resp. $\mathcal{F}^\dagger$ ) is called the \emph{Fock module}  (resp. \emph{dual Fock module}) of $\mathcal{C}$. The Fock space $\mathcal{F}$  decomposes into two subspaces $\mathcal{F}_{\mathrm{even}}$ and $\mathcal{F}_{\mathrm{odd}}$, where $\mathcal{F}_{\mathrm{even}}$(resp. $\mathcal{F}_{\mathrm{odd}}$) is generated by all vectors obtained from $\rvac$ by applying even (resp. odd) numbers of $\phi_n(n\geq 0)$. There exists an anti-algebra automorphism $x\mapsto x^*$ of $\mathcal{C}$ as $\mathbb{Q}(\beta)$-algebra such that $\phi^*_n=(-1)^n\phi_{-n}$ for $n\in \Z$ and $(xy)^*=y^*x^*$.
There is a unique non-degenerate bilinear form $\langle\,\cdot\, |\,\cdot\, \rangle: \mathcal{F}^\dagger\times \mathcal{F}\rightarrow \mathbb{Q}(\beta)$ such that
\begin{align}
 \langle ua  | v \rangle&=\langle u  | av \rangle
\quad \text{for $u\in \mathcal{F}^\dagger,\;
v\in \mathcal{F},\;a\in \mathcal{C},$} \\
\lvac \mathrm{vac} \rangle&=1.
\end{align}
For $u\in \mathcal{F}^\dagger$
and $v\in \mathcal{F}$, we denote
$u$, $v$ by $\langle u|$,
$|v\rangle$ respectively. Moreover, for $a\in \mathcal{C}$, we denote the value
$\langle ua|v\rangle=\langle u|av\rangle$ by $\langle u|a| v\rangle$. In particular, $\lvac   a\rvac$  is called the \emph{vacuum expectation value} of $a\in \mathcal{C}.$

For example, we have
\begin{equation}
\lvac  \phi_m\phi_n\rvac=\begin{cases}
2(-1)^m\delta_{m+n,0}
& (n>0)\\
\delta_{m,0}& (n=0)\\
0 & (n<0)
\end{cases}
\label{eq:phiphi}
\end{equation}

If $A$ is a skew-symmetric matrix
of even size, we denote its Pfaffian by $\mathrm{Pf}(A)$.
\begin{prop}[Wick's theorem, \cite{DJM}, \cite{AZ}]
 For $v_1,\ldots,v_r\in \bigoplus_{n\in \Z}\mathbb{Q}(\beta) \psi_n$,
\begin{align}
\lvac  v_1\cdots v_r\rvac=
\begin{cases}
\mathrm{Pf}(\lvac  v_iv_j\rvac)_{1\le i,j\le r}& \text{if $r$ is even,}\\
0 & \text{if $r$ is odd.}
\end{cases}
\end{align}
\end{prop}


The \emph{neutral fermion field} is defined as $\phi(z)=\sum\limits_{n\in \mathbb{Z}}\phi_n z^n$, where $z$ is a formal variable. We have from \eqref{eq:phiphi}
\begin{align*}
\lvac  \phi(z) \phi(w)\rvac
&=\sum_{m,n\in \Z}
\lvac  \phi_m \phi_n\rvac z^mw^n\\
&=1+\sum_{n=1}^\infty 2(-1)^n z^{-n}w^n
\\
&=\frac{1-w/z}{1+w/z}.\end{align*}
\subsection{$\beta$ deformed Fermions}
We define two $\beta$-\emph{deformed fermion fields}  $\phi^{(\beta)}(z)=\sum\limits_{n\in \mathbb{Z}}\phi^{(\beta)}_n z^n$ and $\phi^{[\beta]}(z)=\sum\limits_{n\in \mathbb{Z}}\phi^{[\beta]}_n z^n$ via the following expansions:
\begin{align*}
&\sum_{n=0}^\infty\phi^{(\beta)}_n z^n=
\sum_{n=0}^\infty\phi_n \left(z+\frac{\beta}{2}\right)^n,\quad
\sum_{n=1}^\infty\phi^{(\beta)}_{-n} z^{-n}=
\sum_{n=1}^\infty\phi_{-n} \left(\frac{z^{-1}}{1+\frac{\beta}{2}z^{-1}}\right)^n,\\
&\sum_{n=1}^\infty\phi^{[\beta]}_n z^n=
\sum_{n=1}^\infty\phi_n \left(\frac{z}{1+\frac{\beta}{2}z}\right)^n,\quad
\sum_{n=0}^\infty\phi^{[\beta]}_{-n} z^{-n}=
\sum_{n=0}^\infty\phi_{-n} \left(z^{-1}+\frac{\beta}{2}\right)^n.
\end{align*}

From the vacuum conditions \eqref{eq:rvac}, \eqref{eq:lvac},
\begin{align}\phi^{[\beta]}_{-n}\rvac=\phi^{(\beta)}_{-n}\rvac=0,\quad \lvac\phi^{[\beta]}_{n}=\lvac\phi^{(\beta)}_{n}=0, \text{for $n>0$.}\label{phibet-vacu}
\end{align}
Then it follows from equation \eqref{phi-relation} that these operators satisfy the anticommutation relation
\begin{align}
[\phi^{(\beta)*}_m,\phi^{[\beta]}_n]_{+}=2\delta_{m,n}+\beta\delta_{m,n-1}.\label{phibeta-relation}
\end{align}
In addition,
\begin{align}
&\lvac\phi^{(\beta)}(u)\phi^{(\beta)}(w)\rvac=\frac{w^{-1}-u^{-1}}{w^{-1}\oplus u^{-1}},\label{Phivalue}\\
&\lvac\phi^{[\beta]*} (u)\phi^{[\beta]*}(w)\rvac=\frac{w-u}{w\oplus u},\label{phivalue}
\end{align}
where $\frac{w^{-1}-u^{-1}}{w^{-1}\oplus u^{-1}}$ and $\frac{w-u}{w\oplus u}$ are respectively expanded as
$$\frac{1-w u^{-1}}{1-w u^{-1}+\beta  u^{-1}}=1-(\beta+2w)u^{-1}+(2w^{2}+3\beta w+\beta^2)u^{-2}+\cdots,$$
and
$$\frac{1-u w^{-1}}{1+u w^{-1}+\beta  u}=1-(\beta+2w^{-1})u+(2w^{-2}+3\beta w^{-1}+\beta^2)u^2+\cdots.$$
 Define another two $\beta$-\emph{deformed fermion fields} 
as
\begin{align}
\Phi^{(\beta)}(z)&:=\sum\limits_{n\in \Z}\Phi_{n}^{(\beta)}z^{n}
=\frac{1}{2+\beta z^{-1}}\phi^{(\beta)}(z),\\
\Phi^{[\beta]}(z)&:=\sum\limits_{n\in \Z}\Phi_{n}^{[\beta]}z^{n}=\frac{1}{2+\beta z}\phi^{[\beta]}(z).
\label{eq:Phi(z) Psi(z)}
\end{align}
For any $n \in \mathbb{Z}$, these $\beta$-deformed fermions  are related to the neutral fermion through formal series expansions:
\begin{align*}
&\Phi_{n}^{(\beta)}=\frac{1}{2}\phi_{n}^{(\beta)}-\frac{\beta}{4}\phi_{n+1}^{(\beta)}+\frac{\beta^2}{8}\phi_{n+2}^{(\beta)}+\cdots,\\
&\Phi_{n}^{[\beta]}=\frac{1}{2}\phi_{n}^{[\beta]}-\frac{\beta}{4}\phi_{n-1}^{[\beta]}+\frac{\beta^2}{8}\phi_{n-2}^{[\beta]}+\cdots,
\end{align*}
Using \eqref{phibet-vacu}, we have the following annihilation rule
\begin{align}
\Phi^{(\beta)*}_{n}\rvac=\Phi^{[\beta]}_{-n}\rvac=0,\quad \lvac\Phi^{(\beta)}_{n}=0\quad \text{for $n>0$.}\label{Phibet-vacu}
\end{align}
These $\beta$-deformed fermions satisfies the following commutation relations
\begin{align}
[\Phi_{n}^{(\beta)*},\phi_{m}^{[\beta]}]_+=[\Phi_{n}^{[\beta]*},\phi_{m}^{(\beta)}]_+=\delta_{m,n}.\label{Phi-phi}
\end{align}

\subsection{$\beta$-deformed Boson-Fermion correspondence}
Let $p_n(x)=\sum\limits_{k>0}x_k^n$ be the $n$-th power sum in $x=(x_1,x_2,\cdots)$. The corresponding $\beta$-deformed power sums $p_n^{(\beta)}$ and $p_n^{[\beta]}$ are defined by:
\begin{align*}
&p_n^{(\beta)}(x)=p_n\left(\frac{x}{1+\beta/2}\right)=\sum_{i=0}^{\infty}\binom{-n}{i}({\beta}/{2})^ip_{n+i}(x),\\
&p_n^{[\beta]}(y)=p_n(y+\beta/{2})-p_n(\beta/2)=\sum_{i=0}^{n}\binom{n}{i}(\beta/2)^ip_{i}(y).
\end{align*}
For each odd integer $n$, we define the following operators acting on $\mathcal{F}$ and $\mathcal{F}^\dagger$
\begin{equation}
b_n=\frac{1}{4}\sum\limits_{i\in\mathbb{Z}}(-1)^i\phi_{-i-m}\phi_{i}.
\end{equation}
They satisfy following the Heisenberg relation, and the commutation relation with neutral Fermion:
\begin{equation}
[b_m,b_n]=\frac{m}{2}\delta_{m+n,0},\quad
[b_m,\phi_n]=\phi_{n-m}.
\end{equation}
 The  $\beta$-\emph{deformed Hamiltonian operators} 
are defined as
\begin{align*}
H^{(\beta)}(x)&=2\sum_{n=1,3,5,\cdots}\frac{p^{(\beta)}_n(x)}{n}b_n,\quad
H^{[\beta]}(y)=2\sum_{n=1,3,5,\cdots}\frac{p^{[\beta]}_n(y)}{n}b_n.
\end{align*}


Define
\begin{align}
\widehat{G\Gamma}(x)
&:=\mathbb{Q}(\beta)[\![p_1^{(\beta)}(x),p_3^{(\beta)}(x),\cdots]\!],\\
g\Gamma(y)
&:=\mathbb{Q}(\beta)[p_1^{[\beta]}(y),p_3^{[\beta]}(y),\cdots]\label{def_g_Gam}.
\end{align}
There exists a perfect pairing
\[
\widehat{G\Gamma}(x)\times g\Gamma(y)\longrightarrow \mathbb{Q}(\beta)
\]
such that for all odd partitions $\lambda$ and $\mu$
\begin{equation}
\langle p_\lambda^{(\beta)}(x), p_\mu^{[\beta]}(y)\rangle
=2^{-\ell(\lambda)}\delta_{\lambda\mu},
\label{eq:pp}
\end{equation}
 For odd partition $\lambda$,
\begin{equation}
p_\lambda^{(\beta)}(x)
:= p_{\lambda_1}^{(\beta)}(x)p_{\lambda_2}^{(\beta)}(x)\cdots,
\quad
p_\lambda^{[\beta]}(y)
:= p_{\lambda_1}^{[\beta]}(y)p_{\lambda_2}^{[\beta]}(y)\cdots.
\end{equation}

We define
\begin{equation}
\hat{\mathcal{F}}_{\mathrm{even}}=
\mathrm{Hom}_{\mathbb{Q}(\beta)}({\mathcal{F}}^\dagger_{\mathrm{even}},\mathbb{Q}(\beta)).
\end{equation}
There is a
linear isomorphism
$\mathcal{F}^\dagger_\mathrm{even}\rightarrow\mathcal{F}_{\mathrm{even}},\; \langle u|\mapsto
|u^*\rangle$ given by
\begin{equation}
|u^*\rangle:=a^*
\rvac \quad
\text{whenever  $\langle u|
=\lvac  a, \quad a\in \mathcal{C}.$}
\end{equation}
\begin{lem}[Boson-Fermion correspondence, \cite{Iwao2023-2}, \S~ 3.4]\label{lem:x1x2}
There exist $\mathbb{Q}(\beta)$-linear isomorphisms such that
\begin{align*}
\hat{\mathcal{F}}_{\mathrm{even}}&\longrightarrow \widehat{G\Gamma}(x),\quad
\langle u|\mapsto \lvac  e^{\mathcal{H}^{(\beta)}(x)}
|u^*\rangle,\\
\mathcal{F}_{\mathrm{even}}&\longrightarrow g\Gamma(y), \quad
|v\rangle\mapsto \lvac  e^{\mathcal{H}^{[\beta]}(y)}
|v\rangle.
\end{align*}
Moreover, via these isomorphisms,
the natural
pairing
$\hat{\mathcal{F}}_{\mathrm{even}}\times \mathcal{F}_{\mathrm{even}}\rightarrow \mathbb{Q}(\beta)$ corresponds to the pairing \eqref{eq:pp} on
$\widehat{G\Gamma}(x)\times g\Gamma(y)\rightarrow \mathbb{Q}(\beta).$
\end{lem}

\subsection{ Free-Fermion formalism of $gq,gp,GP,GQ$--- Non-equivariant case }
We correct  the results
of free-Ferminon formalism
for the non-equivariant
 case.

Let $$\theta=2\sum_{n=1,3,5,\cdots}\left(\frac{\beta}{2}\right)^n\frac{b_n}{n},\quad \theta^*=2\sum_{n=1,3,5,\cdots}\left(\frac{\beta}{2}\right)^n\frac{b_{-n}}{n}.$$
We have
\begin{equation}
e^{\theta}\rvac= \rvac,\quad \lvac  e^{\theta^*}  =\lvac .
\end{equation}
\begin{lem} \label{Lem:beta-theta}We have
\begin{align}
e^{\theta}
\phi^{[\beta]}(z)e^{-\theta}&=(1+\beta z)\phi^{[\beta]}(z),
\label{eq:phi-theta}\\
\label{eq:phi-theta-star}
e^{\theta^*}
\phi^{(\beta)}(z)e^{-\theta^*}&=(1+\beta z^{-1})\phi^{(\beta)}(z).
\end{align}
\end{lem}
\begin{proof}Equation \eqref{eq:phi-theta} is obtained from \cite[Proposition 9.2]{Iwao2023-2} by applying the anti-automorphism $*$. Equation
\eqref{eq:phi-theta-star} is given in \cite[Proposition 4.5]{Iwao2023-2}.
\end{proof}

For a strict partition $\lambda$ of length $r$, we define the following vectors
 \begin{align*}&
 |\lambda\rangle_{gq}=\begin{cases}
    \phi_{\lambda_1}^{[\beta]}e^{-\theta}\phi_{\lambda_2}^{[\beta]}e^{-\theta}\cdots\phi_{\lambda_r}^{[\beta]}e^{-\theta}\rvac, & \text{$r:$ even}\\
\phi_{\lambda_1}^{[\beta]}e^{-\theta}\phi_{\lambda_2}^{[\beta]}e^{-\theta}\cdots\phi_{\lambda_r}^{[\beta]}e^{-\theta}\phi_{0}^{[\beta]}e^{-\theta}\rvac, & \text{$r:$ odd}
\end{cases}\\
&|\lambda\rangle_{GP}=\begin{cases}
    \Phi_{\lambda_1}^{(\beta)}e^{\theta^*}\Phi_{\lambda_2}^{(\beta)}e^{\theta^*}\cdots\Phi_{\lambda_r}^{(\beta)}e^{\theta^*}\rvac, & \text{$r:$ even}\\
\Phi_{\lambda_1}^{(\beta)}e^{\theta^*}\Phi_{\lambda_2}^{(\beta)}e^{\theta^*}\cdots\Phi_{\lambda_r}^{(\beta)}e^{\theta^*}\phi_{0}^{(\beta)}e^{\theta^*}\rvac. & \text{$r:$ odd}
\end{cases}\\
&|\lambda\rangle_{gp}=
    (\Phi_{\lambda_1}^{[\beta]}-\frac{1}{2}(-\frac{\beta}{2})^{\lambda_1})e^{-\theta}(\Phi_{\lambda_2}^{[\beta]}-\frac{1}{2}(-\frac{\beta}{2})^{\lambda_2})e^{-\theta}\cdots(\Phi_{\lambda_r}^{[\beta]}-\frac{1}{2}(-\frac{\beta}{2})^{\lambda_r})e^{-\theta}(\phi_0+1)\rvac,\\
&|\lambda\rangle_{GQ}=\begin{cases}
    \phi_{\lambda_1}^{(\beta)}e^{\theta^*}\phi_{\lambda_2}^{(\beta)}e^{\theta^*}\cdots\phi_{\lambda_r}^{(\beta)}e^{\theta^*}\rvac & \text{$r:$ even}\\
\phi_{\lambda_1}^{(\beta)}e^{\theta^*}\phi_{\lambda_2}^{(\beta)}e^{\theta^*}\cdots\phi_{\lambda_r}^{(\beta)}e^{\theta^*}\phi_{0}^{(\beta)}e^{\theta^*}\rvac & \text{$r:$ odd}
\end{cases}
\end{align*}

\begin{thm}\label{thm:GP-gq}Let $\lambda\in \SP$, we have
\begin{eqnarray}
GP_\lambda(x)&=&
\lvac  e^{\mathcal{H}^{(\beta)}(x)}
|\lambda\rangle_{GP},\quad
GQ_\lambda(x)=
\lvac  e^{\mathcal{H}^{[\beta]}(x)}
|\lambda\rangle_{GQ},\\
gp_\lambda(y)&=&
\lvac  e^{\mathcal{H}^{[\beta]}(y)}
|\lambda\rangle_{gp},\quad\;\;\;\,
gq_\lambda(y)=
\lvac  e^{\mathcal{H}^{[\beta]}(y)}
|\lambda\rangle_{gq}.
\end{eqnarray}
\end{thm}
The functions for
the single row partitions can be computed by  the following generating functions.
\begin{prop}[\cite{Iwao2023-2}]
  Let  $GP(z)=\sum\limits_{m\in \Z}GP_m(x)z^{m},\;
GQ(z)=\sum\limits_{m\in \Z}GQ_m(x)z^{m},$ and $
gp(z)=\sum\limits_{m=0}^\infty gp_m(y)z^m,\;
gq(z)=\sum\limits_{m=0}^\infty gq_m(y)z^m
$. Then we have
\begin{align}
GQ(u^{-1})&=\frac{1}{(1+\beta u)}\prod_{j=1}^\infty\frac{u\oplus x_j}{u\ominus x_j},\quad GP(u^{-1})=\frac{1}{2+\beta u}GQ(u^{-1}),\\
gq(z)&=\prod_{j=1}^\infty\frac{1-y_j\bar{z}}{1-y_jz},\quad gp(z)=\frac{1}{2+\beta z}{(gq(z)+\beta z+1)}.
\label{eq:gp gq single }
\end{align}
\end{prop}
Note that for $m>0$ we have $GQ_{-m}(x)=(-\beta)^m.$

\begin{cor}\label{cor:gq-GP} Let $\lambda\in \SP$. We have
\begin{align*}&GQ_{\lambda}(x)=[u_1^{-\lambda_1}\dots u_r^{-\lambda_r}]\prod_{i=1}^rGQ(u_i^{-1})\prod_{1\leq i<j\leq r}\frac{u_j\ominus u_i}{u_j\oplus u_i},\\
&GP_{\lambda}(x)=[u_1^{-\lambda_1}\dots u_r^{-\lambda_r}]\prod_{i=1}^r GP(u_i^{-1})\prod_{1\leq i<j\leq r}\frac{u_j\ominus u_i}{u_i\oplus u_j},\\
&gq_\lambda(y)=[z_1^{\lambda_1}\dots z_r^{\lambda_r}]\prod_{i=1}^r gq(z_i)\prod_{1\leq i<j\leq r}\frac{z_i\ominus z_j}{z_i\oplus z_j},\\
&gp_\lambda(y)=[z_1^{\lambda_1}\dots z_r^{\lambda_r}]\prod_{i=1}^r \frac{1}{2+\beta z_i}\sum_{\substack{0\leq a\leq r\\i_1<i_2<\cdots<i_a}}(-1)^{r-a}\prod_{\kappa=1}^a\frac{gq(z_{i_\kappa})}{(1+\beta z_{i_\kappa})^{i_\kappa-\kappa}}\prod_{i_1\leq i_\kappa<i_\theta\leq i_a}\frac{z_{i_\kappa}\ominus z_{i_\theta}}{z_{i_\kappa}\oplus z_{i_\theta}}.
\end{align*}
Here, for $GQ_\lambda(x)$, $GP_\lambda(x)$, the right hand side expands in $\mathbb{Q}(\beta)[[x_1,x_2,\cdots]]((u_r))\cdots ((u_1))$, while for $gq_\lambda(y)$, $gp_\lambda(y)$,
 the right hand side expands in $\mathbb{Q}(\beta)[y_1,y_2,\cdots]((z_1))\cdots ((z_r))$,
 and $[u_1^{-\lambda_1}\dots u_r^{-\lambda_r}]F(u_1,\ldots,u_r)$ denotes the coefficient of
 $u_1^{-\lambda_1}\dots u_r^{-\lambda_r}$
 in a series $F(u_1,\ldots,u_r)$ and the symbol $[z_1^{\lambda_1}\dots z_r^{\lambda_r}]$ has a similar meaning.
\end{cor}


\subsection{ $GQ_{\lambda}(x\,|\,b)$ and $GP_{\lambda}(x\,|\,b)$ as a vacuum expectation value}\label{sec:FacGPGQ}

Introduce the following new $\beta$-deformed fermion fields containing equivariant parameters as
\begin{align*}
\Phi^{(\beta)(k)}(z)&=\prod_{j=1}^{k-1}\frac{1-zb_j}{1+\beta b_j}\Phi^{(\beta)}(z)
    =\sum_{i\in \mathbb{Z}}\Phi_i^{(\beta)(k)}z^i,\\
    \phi^{(\beta)(k)}(z)&=\prod_{j=1}^{k-1}\frac{1-zb_j}{1+\beta b_j}\phi^{(\beta)}(z)
    =\sum_{i\in \mathbb{Z}}\phi_i^{(\beta)(k)}z^i,
\end{align*}
where $k>1$. We define that $\Phi^{(\beta)(k)}(u) = \phi^{(\beta)(k)}(z)=\phi^{(\beta)}(z)$ for all $k \leq 0$. In the specific case of $k=1$, we also allow the notation $\Phi^{(\beta)(1)}(z) = \Phi^{(\beta)}(z)$ and $\phi^{(\beta)(1)}(z) = \phi^{(\beta)}(z)$.   It can be directly calculated that for $k\geq1$
\begin{align}
&\Phi_i^{(\beta)(k)}=\prod_{j=1}^{k-1}\frac{1}{1+\beta b_j}\sum_{j=0}^{k-1}(-1)^je_j(b_1\dots,b_k)\Phi^{(\beta)}_{i-j},\label{Phi-betak}\\
&\phi_i^{(\beta)(k)}=\prod_{j=1}^{k-1}\frac{1}{1+\beta b_j}\sum_{j=0}^{k-1}(-1)^je_j(b_1\dots,b_k)\phi^{(\beta)}_{i-j}.\label{phi-betak}
\end{align}

\begin{lem}[cf. \cite{HIMN}, Lemma 5.19]
The factorial $GP$-function $GP_{\lambda}(x\,|\,b)$ and factorial $GQ$-function $GQ_{\lambda}(x\,|\,b)$
are given by the following formula
\begin{align}
&&GP_{\lambda}(x\,|\,b)=[u_1^{-\lambda_1}\dots u_r^{-\lambda_r}]\prod_{i=1}^r\prod_{j=1}^{\lambda_i-1} GP(u^{-1}_i)
\frac{u_i-b_j}{u_i(1+\beta b_j)}\prod_{1\leq i<j\leq r}\frac{u_j\ominus u_i}{u_i\oplus u_j},\label{GPxbexpand}\\
&&GQ_{\lambda}(x\,|\,b)=[u_1^{-\lambda_1}\dots u_r^{-\lambda_r}]\prod_{i=1}^r\prod_{j=1}^{\lambda_i-1} GQ(u^{-1}_i)
\frac{u_i-b_j}{u_i(1+\beta b_j)}\prod_{1\leq i<j\leq r}\frac{u_j\ominus u_i}{u_i\oplus u_j}.
\end{align}
\end{lem}
\begin{proof}
The formula \cite[Definition 3.5, Definition 3.7]{HIMN2} reads that for $m\ge 0$,
\begin{align}
GP_m(x)&= [u^{-m}]GP(u^{-1}),\\
GP_m^{(\ell)}(x\,|\,b)&=[u^{-m}]GP(u^{-1})
\prod_{i=1}^\ell
\frac{u-b_i}{u(1+\beta b_i)}\quad \text{for $\ell\ge 0$},
\label{GPm-l generating}
\end{align}
where $GP_m^{(\ell)}(x\,|\,b)$ is the equivariantly shifted
$GP$-function which is defined by
\begin{equation}
  GP_m^{(\ell)}(x\,|\,b)=\prod_{i=1}^\ell(1+\beta b_i)^{-1}
  \sum_{i=0}^{\ell}(-1)^ie_i(b_1,\ldots,b_\ell)GP_{m+i}(x).
\end{equation}
Acoording to a general result on Pfaffian (Proposition \ref{prop:Knuth}) below,
one sees that the equation
\eqref{GPxbexpand} is equivalent to
the Pfaffian formula \cite[Definition 3.13]{HIMN2} for $GP_\lambda(x\,|\,b)$.
\end{proof}

For any arbitrary strict partition $\lambda=(\lambda_1\dots,\lambda_r)$, we denote $r'$ by the minimum even integer that is equal to or greater than  $r$, if $r'\neq r$, let $\lambda_{r'}=0$,  define

\begin{align*}
|\lambda\rangle_{(GP,b)}&:=
\Phi_{\lambda_1}^{(\beta)(\lambda_1)}e^{\theta^*}\Phi_{\lambda_2}^{(\beta)(\lambda_2)}e^{\theta^*}\cdots\Phi_{\lambda_{r'}}^{(\beta)(\lambda_{r'})}e^{\theta^*}\rvac\\
&=\begin{cases}
    \Phi_{\lambda_1}^{(\beta)(\lambda_1)}e^{\theta^*}\Phi_{\lambda_2}^{(\beta)(\lambda_2)}e^{\theta^*}\cdots\Phi_{\lambda_r}^{(\beta)(\lambda_r)}e^{\theta^*}\rvac, & \text{$r:$ even}\\
\Phi_{\lambda_1}^{(\beta)(\lambda_1)}e^{\theta^*}\Phi_{\lambda_2}^{(\beta)(\lambda_2)}e^{\theta^*}\cdots\Phi_{\lambda_r}^{(\beta)(\lambda_r)}e^{\theta^*}\phi^{(\beta)}_{0}e^{\theta^*}\rvac, & \text{$r:$ odd}\end{cases}
\end{align*}
\begin{align*}
|\lambda\rangle_{(GQ,b)}&:=
\phi_{\lambda_1}^{(\beta)(\lambda_1)}e^{\theta^*}\phi_{\lambda_2}^{(\beta)(\lambda_2)}e^{\theta^*}\cdots\phi_{\lambda_{r'}}^{(\beta)(\lambda_{r'})}e^{\theta^*}\rvac\\
&=\begin{cases}
    \phi_{\lambda_1}^{(\beta)(\lambda_1)}e^{\theta^*}\phi_{\lambda_2}^{(\beta)(\lambda_2)}e^{\theta^*}\cdots\phi_{\lambda_r}^{(\beta)(\lambda_r)}e^{\theta^*}\rvac, & \text{$r:$ even}\\
\phi_{\lambda_1}^{(\beta)(\lambda_1)}e^{\theta^*}\phi_{\lambda_2}^{(\beta)(\lambda_2)}e^{\theta^*}\cdots\phi_{\lambda_r}^{(\beta)(\lambda_r)}e^{\theta^*}\phi^{(\beta)}_{0}e^{\theta^*}\rvac. & \text{$r:$ odd}\end{cases}
\end{align*}

We extend the natural pairing on $\mathcal{F}^\dagger \times \mathcal{F}$
by regarding the $b_i$ and $x_i,y_i$ as scalars.

\begin{thm}\label{thm:GP as VEV}
For any arbitrary strict partition $\lambda=(\lambda_1\dots,\lambda_r)$, the factorial $GP$-function $GP_{\lambda}(x\,|\,b)$ and the factorial $GQ$-function $GQ_{\lambda}(x\,|\,b)$ can be  expressed as
\begin{align}\label{GPxb}
GP_\lambda(x\,|\,b)=
\lvac e^{\mathcal{H}^{(\beta)}(x)}
|\lambda\rangle_{(GP,b)},\quad
GQ_\lambda(x\,|\,b)=
\lvac e^{\mathcal{H}^{(\beta)}(x)}
|\lambda\rangle_{(GQ,b)}.
\end{align}
\end{thm}
\begin{proof}
It follows from the definition of $\Phi^{(\beta)(k)}(u)$ and Corollary \ref{cor:gq-GP} that the right-hand side of \eqref{GPxb} is equal to
\begin{align*}
&[u_1^{-\lambda_1}\cdots u_{r'}^{-\lambda_{r'}}]\langle
\mathrm{vac}|e^{H^{(\beta)}(x)}\Phi^{(\beta)(\lambda_1)}(u^{-1}_1) e^{\theta^*}\cdots\Phi^{(\beta)(\lambda_{r'})}(u^{-1}_{r'})e^{\theta^*}
|\mathrm{vac}\rangle,\\
&=[u_1^{-\lambda_1}\cdots u_{r'}^{-\lambda_{r'}}]\prod_{i=1}^{r}\prod_{j=1}^{\lambda_i-1}\frac{u_i-b_j}{u_i(1+\beta b_j)}\langle
\mathrm{vac}|e^{H^{(\beta)}(x)}\Phi^{(\beta)}(u^{-1}_1) e^{\theta^*}\cdots\Phi^{(\beta)}(u^{-1}_{r'})e^{\theta^*}
|\mathrm{vac}\rangle\\
&=[u_1^{-\lambda_1}\cdots u_{r}^{-\lambda_{r}}]\prod_{i=1}^{r}\prod_{j=1}^{\lambda_i-1}\frac{u_i-b_j}{u_i(1+\beta b_j)}GP(u_i^{-1})\prod_{1\leq i<j\leq r}\frac{u_j- u_i}{u_i\oplus u_j},
  \end{align*}
  which is  consistent with \eqref{GPxbexpand}.
  The proof for $GQ_\lambda(x\,|\,b)$ works similarly.
 \end{proof}

\subsection{  $gq_{\lambda}(x\,|\,b)$ as a vacuum expectation value}\label{se:fac-gq}

Introduce the following new $\beta$-deformed fermion field containing equivariant parameters as
\begin{align*}
    \phi^{[\beta](k)}(z)&=\prod_{j=1}^k\frac{z}{z-b_j}\prod_{j=1}^{k-1}(1+\beta b_j)\phi^{[\beta]}(z)=\sum_{i\in \mathbb{Z}}{\phi}_i^{[\beta](k)}z^i,
\end{align*}
where $k>0$,  we make the convention that $\phi^{[\beta](k)}(z)=\phi^{[\beta]}(z)$ for  $k\leq0$. It can be directly calculated that
\begin{align}
&{\phi}_i^{[\beta](k)}=\prod_{\ell=1}^{k-1}(1+\beta b_\ell)\sum_{j=0}^\infty h_j(b_1\dots,b_k)\phi_{i+j}^{[\beta]}.
\end{align}
For any arbitrary strict partition $\lambda=(\lambda_1\dots,\lambda_r)$, we denote $r'$ by the minimum even integer that is equal to or greater than  $r$, if $r'\neq r$, let $\lambda_{r'}=0$,  define
\begin{align*}|\lambda\rangle_{(gq,b)}&:={\phi}^{[\beta](\lambda_1)}_{\lambda_1}e^{-\theta}{\phi}^{[\beta](\lambda_2)}_{\lambda_2}e^{-\theta}\cdots{\phi}^{[\beta](\lambda_{r'})}_{\lambda_{r'}}e^{-\theta}\rvac\\
=&\begin{cases} {\phi}^{[\beta](\lambda_1)}_{\lambda_1}e^{-\theta}{\phi}^{[\beta](\lambda_2)}_{\lambda_2}e^{-\theta}\cdots{\phi}^{[\beta](\lambda_r)}_{\lambda_r}e^{-\theta}\rvac, & \text{$r:$ even}\\
{\phi}^{[\beta](\lambda_1)}_{\lambda_1}e^{-\theta}{\phi}^{[\beta](\lambda_2)}_{\lambda_2}e^{-\theta}\cdots{\phi}^{[\beta](\lambda_r)}_{\lambda_r}e^{-\theta}\phi_{0}^{[\beta]}e^{-\theta}\rvac. & \text{$r:$ odd}\end{cases}
\end{align*}

 The proof of the following key lemma
 will be given in the next subsection.

 \begin{lem}[Duality Lemma]\label{lem:dual GP vs gq}
For any arbitrary strictly  partitions $\lambda$ and $\mu$,
\begin{align}
_{(GP,b)}\langle\mu|\lambda\rangle_{(gq,b)}=\delta_{\mu\lambda}\label{lambda-mu},
\end{align}
where ${}_{(GP,b)}\langle\mu|=(|\mu\rangle_{(GP,b)})^*$.
\end{lem}
\begin{thm}\label{thm: gq VEV}
For $\lambda\in \SP$, we have
   $$gq_{\lambda}(y\,|\,b)=\lvac  e^{\mathcal{H}^{[\beta]}(y)}|\lambda\rangle_{(gq,b)}.$$
\end{thm}
 \begin{proof}
From Lemma \ref{lem:x1x2},
Lemma \ref{lem:dual GP vs gq}, and Theorem \ref{thm:GP as VEV} we deduce that the family of functions
$\{\lvac  e^{\mathcal{H}^{[\beta]}(y)}|\lambda\rangle_{(gq,b)}\mid \lambda\in \SP\}$ is dual to
$\{GP_\lambda(x\,|\,b)\mid\lambda\in \SP\}.$
\end{proof}
\begin{cor}\label{prop:GQ(x|b)}
\label{cor:gen-fac-gq}
The factorial $gq$-functions
are given by the following formula
\begin{equation}
gq_{\lambda}(y\,|\,b)=[z_1^{\lambda_1}\dots z_r^{\lambda_r}]\prod_{i=1}^r gq^{(\lambda_i)}(z_i)\prod_{1\leq i<j\leq r}\frac{z_i\ominus z_j}{z_i\oplus z_j},\label{gqxbvac}
\end{equation}
where
\begin{align*}
gq^{(\lambda_i)}(z_i)=\prod_{j=1}^{\lambda_i}\prod_{k=1}^{\lambda_i-1}\frac{z_i(1+\beta b_k)}{z_i-b_j} gq(z_i),\quad  gq(z_i)=\prod_{j}\frac{1-y_j\bar{z_i}}{1-y_jz_i}.
\end{align*}
\end{cor}
\begin{proof}
It follows from Corollary \ref{cor:gq-GP},  the right hand side of \eqref{gqxbvac} can be rewritten as
 \begin{align*}
&gq_{\lambda}(y\,|\,b)=[z_1^{\lambda_1}\cdots z_{r'}^{\lambda_{r'}}]\langle
\mathrm{vac}|e^{H^{[\beta]}(y)}\phi^{[\beta](\lambda_1)}(z_1) e^{-\theta}\cdots\phi^{[\beta](\lambda_{r'})}(z_{r'}) e^{-\theta}
|\mathrm{vac}\rangle \\
=&[z_1^{\lambda_1}\cdots z_{r}^{\lambda_{r'}}]\prod_{i=1}^{r}\prod_{j=1}^{\lambda_i}\prod_{k=1}^{\lambda_i-1}\frac{z_i(1+\beta b_k)}{z_i-b_j}\langle
\mathrm{vac}|e^{H^{[\beta]}(y)}\phi^{[\beta]}(z_1) e^{-\theta}\cdots\phi^{[\beta]}(z_{r'}) e^{-\theta}
|\mathrm{vac}\rangle\\
=&[z_1^{\lambda_1}\cdots z_{r}^{\lambda_{r}}]\prod_{i=1}^{r}\prod_{j=1}^{\lambda_i}\prod_{k=1}^{\lambda_i-1}\frac{z_i(1+\beta b_k)}{z_i-b_j} gq(z_i)\prod_{1\leq i<j\leq r}\frac{z_i\ominus z_j}{z_i\oplus z_j},
\end{align*}
where $gq(z_i)=\prod_{j\geq1}\frac{1-x_j\bar{z_i}}{1-x_jz_i}$ and $r'$ is the minimum even integer that is equal to or greater than  $r$, if $r'\neq r$, let $\lambda_{r'}=0$.
 \end{proof}



\begin{remark}
For the equivariant homology version of $gp_\lambda(y|b)$ the analog of above Corollary was obtained by Naruse \cite{Naruse}.
\end{remark}

 \subsection{Pfaffian formula for $gq_\lambda(y\,|\,b)$}
 \label{sec:Pf_gq}
 \label{ssec:Pf gq}





To prove a Pfaffian formula for the factorial
$gq$-function, we first state the following proposition concerning the Pfaffian.
\begin{prop}\label{prop:Knuth}Let $\{G_m^{i}\}$ be a set of
commuting variables indexed by $m\in \Z$ and $i\in \{1,\ldots,r\}$. Let
$G^{i}(u)=\sum_{m\in \Z}G_m^{i} u^{-m}$ with $u$ an indeterminate.
Then, for any strict partition $\lambda=(\lambda_1> \cdots> \lambda_r> 0)$ with $r$ even, we have
\begin{equation}\mathrm{Pf}(G_{\lambda_i,\lambda_j}^{i,j})=
[u_1^{-\lambda_1}\cdots u_r^{-\lambda_r}]
\prod_{i=1}^r G^{i}(u_i)\prod_{1\le i<j\le r}
\frac{u_j- u_i}{u_j\oplus u_i},\label{eq:Pf F}
\end{equation}
where
\begin{equation}
G_{\lambda_i,\lambda_j}^{i,j}
=[u_i^{-\lambda_i}u_j^{-\lambda_j}]
G^{i}(u_i)G^{j}(u_j)
\frac{u_j- u_i}{u_j\oplus u_i}.
\end{equation}
\end{prop}
\begin{proof} From a formula due to Knuth \cite{Knu}, the
right hand side of \eqref{eq:Pf F} equals to
\begin{align*}
&[u_1^{-\lambda_1}\cdots u_r^{-\lambda_r}]
\prod_{i=1}^r G^{i}(u_i)\;
\mathrm{Pf}\left(
\frac{u_j- u_i}{u_j\oplus u_i}\right)_{1\le i<j\le r}\\
&=[u_1^{-\lambda_1}\cdots u_r^{-\lambda_r}]\;
\mathrm{Pf}\left(G^{i}(u_i)G^{j}(u_j)
\frac{u_j- u_i}{u_j\oplus u_i}\right)_{1\le i<j\le r}.
\end{align*}
Let us denote
$G_{ij}=G^{i}(u_i)G^{j}(u_j)
{(u_j- u_i)}/(u_j\oplus u_i)$. Each term of
the Pfaffian is of the form
\begin{equation}
\pm G_{i_1i_2}\cdots G_{i_{r-1}i_r},\label{eq:FFFF}
\end{equation}
where $\{1,\ldots,r\}=\{i_1,i_2\}\sqcup\{i_3,i_4\}\sqcup \cdots\sqcup\{i_{r-1},i_r\}$. For each pair $(u_{i_k}, u_{i_{k+1}})$ of variables,
the factor $G_{i_ki_{k+1}}$ is the only function that depends on $u_{i_k}, u_{i_{k+1}}$ among the factors of \eqref{eq:FFFF}.
Therefore we have
\begin{equation}
[u_1^{-\lambda_1}\cdots u_r^{-\lambda_r}]\left(\pm G_{i_1i_2}\cdots G_{i_{r-1}i_r}\right)
=\pm \prod_{k=0}^{r/2-1}
[u_{i_{2k+1}}^{-\lambda_{2k+1}}u_{i_{2k+2}}^{-\lambda_{i_{2k+2}}}]G_{i_{2k+1}i_{2k+2}},
\end{equation}
and we have the desired result.
\end{proof}

Let
\begin{align*}
    gq^{(k)}(z)&:=\prod_{j=1}^k\frac{z}{z-b_j}\prod_{j=1}^{k-1}(1+\beta b_j)gq(z)
=\sum_{i=0}^\infty gq_i^{(k)}(y\,|\,b)z^i,
\end{align*}
then $gq_i^{(k)}(y\,|\,b)=\prod\limits_{j=1}^{k-1}(1+\beta b_j)\sum\limits_{j=0}^\infty h_j(b_1,\dots,b_k)gq_{i+j}(y)$.

From Proposition \ref{prop:Knuth} and Corollary \ref{prop:GQ(x|b)} we
immediately obtain
the Pfaffian formula for
$gq_\lambda(y\,|\,b).$
\begin{cor}\label{cor:Pf gq}
For $\lambda\in \SP$, if r be the smallest even integer greater than or equal
to $\ell(\lambda)$ then
   $$gq_{\lambda}(y\,|\,b)=\mathrm{Pf}\Big(\sum_{q=0}^{\lambda_j}\sum_{p=-q}^{\lambda_i} a_{p,q}^{i,j}gq_{\lambda_i-p}^{(\lambda_i)}(y\,|\,b)gq_{\lambda_j-q}^{(\lambda_j)}(y\,|\,b)\Big)_{r\times r},$$
%
where $a_{p,q}^{i,j}\in \mathbb{Q}(\beta)$ defined by the following expansion in $\mathbb{Q}(\beta)((z_i))((z_j))$
\begin{equation*}
    \frac{1}{(1+\beta z_i)^{i-1}(1+\beta z_j)^{j-1}}\frac{z_i-z_j}{z_i\oplus z_j}=\sum_{q=0}^\infty\sum_{p=-q}^\infty a_{p,q}^{i,j}z_i^{p}z_j^q,
\end{equation*}
\end{cor}



\subsection{Jacobi-Trudi type formula of the
expansion coefficiants }
\label{sec:JT}

Now, consider the expansion of $gq_\lambda(y\,|\,b)$ and $GP_\mu(x\,|\,b)$ as in Theorem \ref{thm:gq(xb)}.
\begin{thm}\label{thm:gq(xb)}
The following expansion holds:
\begin{align}
gq_\lambda(y\,|\,b) &= \sum_{\mu: \mu \supset \lambda,\;
\ell(\mu)=\ell(\lambda)}
c_{\lambda\mu}^{gq}(b\,|\,c) gq_\mu(y\,|\,c),\label{exp_eq_gq}\\
GP_\mu(x\,|\,c) &= \sum_{\lambda: \lambda \subset \mu,\;\ell(\lambda)=\ell(\mu)}
c_{\mu\lambda}^{GP}(c\,|\,b) GP_\lambda(x\,|\,b)\label{exp_eq_GP},
\end{align} where
\begin{align}
&c_{\lambda\mu}^{gq}(b\,|\,c)
=c_{\mu\lambda}^{GP}(c\,|\,b) \\
&=\frac{d_{\lambda}(b)}{d_{\mu}(c)}
\det \left(
\sum_{k=0}^\infty\beta^k
\binom{i-j}{k}
h_{\mu_i-\lambda_j+k}
(b_1,\ldots,b_{\lambda_j};-c_1,\ldots,-c_{\mu_i-1})
\right).
\end{align}
\end{thm}
\begin{proof}
 By means of  Lemma \ref{lem:x1x2}, Lemma \ref{lem:dual GP vs gq}  and  Theorem \ref{thm: gq VEV}, the coefficient $C_{\lambda,\mu}(b\,|\,c)$ can be expressed as
\begin{align*}
&C_{\lambda\mu}(b\,|\,c)= \langle GP_\mu(x\,|\,c),gq_\lambda(y\,|\,b)\rangle \\
&=\langle \langle \mathrm{vac} |e^{\mathcal{H}^{(\beta)}}
|\mu\rangle_{(GP,c)},\langle \mathrm{vac} |e^{\mathcal{H}^{[\beta]}}
|\lambda\rangle_{(gq,b)}\rangle\\
&= {}_{(GP,c)}\langle\mu|
\lambda\rangle_{(gq,b)}.
\end{align*}
Hence by using Theorem \ref{thm:b-c}, we can find that $\mu \supset \lambda$ with $\ell(\mu) = \ell(\lambda)$.
\end{proof}

The formula
\eqref{exp_eq_GP} generalizes the result of Ivanov
\cite[Theorem 10.2]{Iv}
for the factorial $Q$-functions.
\begin{cor}\label{cor:gq}
For any partition $\lambda $, the following expansion holds:
\begin{align}
gq_\lambda(x\,|\,b) &= \sum_{\mu: \mu \supset \lambda,\;
\ell(\lambda)=\ell(\mu)}
C_{\lambda\mu}^{gq}(b) gq_\mu(x), \label{eq_exp_eq_gq}\\
GP_\lambda(x\,|\,b) &= \sum_{\mu: \mu\subset \lambda,\;
\ell(\mu)=\ell(\lambda)}
C_{\lambda\mu}^{GP}(b) GP_\mu(x),\label{eq_exp_eq_GP}
\end{align}
where
\begin{align*}
C_{\lambda\mu}^{gq}(b)&=
d_{\lambda}(b)
\det\left(
\sum_{k=0}^{\infty}
\beta^k
\binom{i-j}{k}h_{\mu_i-\lambda_j+k}(b_1,\ldots,b_{\lambda_j})
\right)_{1\le i,j\le r},\\
C_{\lambda\mu}^{GP}(b)&=
d_\lambda(b)^{-1}
\det\left(
\sum_{k=0}^{\infty}
\beta^k
\binom{i-j}{k}e_{\lambda_i-\mu_j+k}(-b_1,\ldots,-b_{\lambda_i-1})
\right)_{1\le i,j\le r}.
\end{align*}
\end{cor}

\subsection{A formula on contour integral}

In this subsection, we consider contour integrals on the complex plane of the form
\[
\oint_{|x|<|y|<|z|}f(x,y,z)dy,
\]
which denote the line integral along a circle centered at $y=0$ with radius greater than $|x|$ but smaller than $|z|$.

\begin{lem}\label{lem:integral2}
Let $A,B\in \mathbb{Z}$, $m>0$ and $n,d\geq 0$.
Then we have
\begin{equation}\label{eq:rational_expression2}
\begin{aligned}
&[u^{m+d} w^n]
\left(
{(1+\beta w)^A(1+\beta u^{-1})^B}
\frac{\prod_{k=1}^{m-1}(1-uc_k)}{\prod_{k=1}^n(1-w^{-1}b_k)}
\frac{uw}{1-uw}
\right)
\\
&=
\sum_{l\geq 0}\binom{A+B}{l}\beta^lh_{l-n+m+d}(b_1,\ldots,b_n;-c_1,\ldots,-c_{m-1})
\end{aligned}
\end{equation}
where
$h_p(b_1,\dots,b_n;-c_1,\dots,-c_{m-1})=0$ if $p<0$ and
$
\binom{-n}{l}=(-1)^l\binom{n+l-1}{l}
$.
\end{lem}
\begin{proof}
Let $I$ be the expression on the left-hand side in \eqref{eq:rational_expression2}.
Then we have
\[
\begin{aligned}
I&=
\frac{1}{(2\pi i)^2}\oint\!\!\!\oint_{|b|,|c|<|w|<|u|^{-1}<|\beta|^{-1}}
{(1+\beta w)^A(1+\beta u^{-1})^B}
\frac{\prod_{k=1}^{m-1}(1-uc_k)}{\prod_{k=1}^n(1-w^{-1}b_k)}
\frac{uw}{1-uw}\frac{dudw}{u^{m+d+1}w^{n+1}}\\
&\stackrel{t=u^{-1}}{=}
\frac{-1}{(2\pi i)^2}\oint\!\!\!\oint_{|b|,|c|<|w|<|t|<|\beta|^{-1}}
{(1+\beta w)^A(1+\beta t)^B}
\frac{\prod_{k=1}^{m-1}(1-t^{-1}c_k)}{\prod_{k=1}^n(1-w^{-1}b_k)}
\frac{w}{t-w}\frac{d(t^{-1})dw}{t^{-m-d-1}w^{n+1}}\\
&\stackrel{d(t^{-1})=-t^{-2}dt}{=}
\frac{1}{(2\pi i)^2}\oint\!\!\!\oint_{|b|,|c|<|w|<|t|<|\beta|^{-1}}
{(1+\beta w)^A(1+\beta t)^Bt^d}
\frac{\prod_{k=1}^{m-1}(t-c_k)}{\prod_{k=1}^n(w-b_k)}
\frac{dtdw}{t-w}.
\end{aligned}
\]
Applying the Residue theorem for the contour integral $\oint dt$, we obtain
\[
I=
\frac{1}{2\pi i}\oint_{|b|,|c|<|w|<|\beta|^{-1}}
(1+\beta w)^{A+B}w^d
\frac{\prod_{k=1}^{m-1}(w-c_k)}{\prod_{k=1}^n(w-b_k)}
dw.
\]
Expanding the integrand over the domain $\{|b|,|c|<|w|<|\beta|^{-1}\}$, we obtain
\[
\begin{aligned}
I&=
\frac{1}{2\pi i}\oint_{|b|,|c|<|w|<|\beta|^{-1}}
\left[(1+\beta w)^{A+B}w^d\left\{\sum_{i=0}^\infty (-1)^ih_i(b_1,\ldots,b_n;-c_1,\ldots,-c_{m-1})w^{-n+m-i-1}
\right\}
\right]dw\\
&=\sum_{l\geq 0}\binom{A+B}{l}\beta^lh_{l-n+m+d}(b_1,\ldots,b_n;-c_1,\ldots,-c_{m-1}).
\end{aligned}
\]
\end{proof}

 \subsection{Proof on General expansion formulas}


\begin{lem}\label{Lem:phi-Phi}
For $n>0$ and  $m\in \mathbb{Z}$, we have
\begin{align*}
&\lvac e^{-\theta^*}\phi_{m}^{[\beta]*}\Phi^{(\beta)}_{n}e^{\theta^*}\rvac=\delta_{n,m}.
\end{align*}

\end{lem}
\begin{proof}
In fact for $n>0$, it follows from \eqref{phibet-vacu}, \eqref{Phibet-vacu}, \eqref{Phi-phi} and Lemma \ref{Lem:beta-theta} that
\begin{align*}
&\lvac e^{-\theta^*}\phi_{0}^{[\beta]*}\Phi^{(\beta)}_{n}e^{\theta^*}\rvac=\lvac e^{-\theta^*} \phi_{0}^{[\beta]*}e^{\theta^*}e^{-\theta^*}\Phi^{(\beta)}_{n}e^{\theta^*}\rvac\\
&= \lvac( \phi_{0}^{[\beta]*}+\beta\phi_{-1}^{[\beta]*})(\Phi^{(\beta)}_{n}+\beta\Phi^{(\beta)}_{n+1}+\beta^2\Phi^{(\beta)}_{n+2}+\cdots)\rvac=0.
\end{align*}
If we impose $m\neq 0$, then
 \begin{align*}
&\lvac e^{-\theta^*}\phi_{m}^{[\beta]*}\Phi^{(\beta)}_{n}e^{\theta^*}\rvac=\lvac\phi_{m}^{[\beta]*}\Phi^{(\beta)}_{n}e^{\theta^*}\rvac\delta_{m>0}\\
&= \delta_{n,m}\delta_{m>0}-\lvac \Phi^{(\beta)}_{n}\phi_{m}^{[\beta]*}e^{\theta^*}\rvac\delta_{m>0}=\delta_{n,m}\delta_{m>0}.
\end{align*}
\end{proof}
\begin{lem}\label{Lem:phi-phi}
For $m\in \mathbb{Z}$, we have
\begin{align*}
&\lvac e^{-\theta^*}\phi_{m}^{[\beta]*}\phi^{(\beta)}_{0}e^{\theta^*}\rvac=\delta_{m,0}.
\end{align*}
\end{lem}
\begin{proof}
For $m<0$, the lemma can obviously be derived.
If $m>1$, by using \eqref{phibet-vacu},
\eqref{phibeta-relation} and Lemma \ref{Lem:beta-theta}, we can obtain
\begin{align*}
&\lvac e^{-\theta^*}\phi_{m}^{[\beta]*}\phi^{(\beta)}_{0}e^{\theta^*}\rvac=-\lvac \phi^{(\beta)}_{0}\phi_{m}^{[\beta]*}e^{\theta^*}\rvac=-\lvac\phi^{(\beta)}_{0}e^{\theta^*}(\phi_{m}^{[\beta]*}+\beta\phi_{m-1}^{[\beta]*})\rvac=0.
\end{align*}
If $m=0$,
\begin{align*}
&\lvac e^{-\theta^*}\phi_{0}^{[\beta]*}\phi^{(\beta)}_{0}e^{\theta^*}\rvac\\
&=\lvac (\phi_{0}^{[\beta]*}+\beta\phi_{-1}^{[\beta]*})(\phi^{(\beta)}_{0}+\beta\phi^{(\beta)}_{1}+\beta^2\phi^{(\beta)}_{2}+\cdots)\rvac\\
&=\lvac\phi_{0}^{[\beta]*}\phi^{(\beta)}_{0}\rvac=\lvac\phi_{0}^{*}\phi_{0}\rvac=1.
\end{align*}
If $m=1$, it follows from \eqref{phibet-vacu}, \eqref{phibeta-relation} that
\begin{align*}
&\lvac e^{-\theta^*}\phi_{1}^{[\beta]*}\phi^{(\beta)}_{0}e^{\theta^*}\rvac=\beta-\lvac e^{-\theta^*}\phi^{(\beta)}_{0}\phi_{1}^{[\beta]*}e^{\theta^*}\rvac\\
&=\beta-\lvac (\phi^{(\beta)}_{0}+\beta\phi^{(\beta)}_{1}+\beta^2\phi^{(\beta)}_{2}+\cdots)(\phi_{1}^{[\beta]*}+\beta\phi_{0}^{[\beta]*})\rvac\\
&=\beta-\beta\lvac\phi^{(\beta)}_{0}\phi_{0}^{[\beta]*}\rvac=0.
\end{align*}
\end{proof}

Let $\lambda\in\SP.$ Define
\begin{equation}\label{eq_d_lam}
d_\lambda(b):=
\prod_{i=1}^{\ell(\lambda)}\prod\limits_{l=1}^{\lambda_i-1}(1+\beta b_l).
\end{equation}
 \begin{prop}\label{thm:b-c}
 Let $\lambda,\mu$ be strict partitions.
 Let $c=(c_1,c_2,\ldots)$ a parameter sequence.
Then ${}_{(GP,c)}\langle\mu|
\lambda\rangle_{(gq,b)}$ is non-zero only if $\lambda,\mu$
 have the same length, say  $r$.
 If this holds we have
 \begin{align}
 \label{eq:pair c b}{}_{(GP,c)}\langle\mu|
\lambda\rangle_{(gq,b)}=
d_{\lambda}(b)d_{\mu}(c)^{-1}\det({D}),
\end{align}
where ${D}$ is an $r\times r$ matrix defined by
\begin{equation}
D_{ij}=\sum_{k=0}^\infty
\beta^k \binom{i-j}{k}
h_{\mu_i-\lambda_j+k}(b_1,\ldots,b_{\lambda_j};-c_1,\ldots,-c_{\mu_i-1}),
\label{eq:D in h}
\end{equation}
where
$\binom{-n}{k}$ denotes the generalized binomial
 coefficients given by
$(-1)^k\binom{n+k-1}{k}$ for $n>0$ and $k\ge 0$,
\end{prop}
\begin{proof}
First we assume that the length $r$ of the strictly partition $\mu$  is even. It follows from the definition of $|\mu\rangle_{(GP,b)}$ that
 \begin{align*}
{}_{(GP,c)}\langle\mu|=(|\mu\rangle_{(GP,c)})^*=\lvac e^\theta\Phi_{\mu_{r}}^{(\beta)*(\mu_{s})}\cdots e^\theta\Phi_{\mu_1}^{(\beta)*(\mu_1)}
\end{align*}
 where we simply
denote  $\Phi_{\mu_i}^{(\beta)*(\mu_i)}$ the
equivariantly deformed
Fermion with equivariant parameter $c=(c_1,c_2,\ldots).$

Let us denote $
u^\mu w^\lambda=
u_1^{\mu_1}\cdots u_r^{\mu_r}
w_1^{\lambda_1}
\cdots w_{s'}^{\lambda_{s'}}.
$
We have
  \begin{align}
&{}_{(GP,c)}\langle\mu|
\lambda\rangle_{(gq,b)}\nonumber\\
&=[u^{\mu}
w^{\lambda}]
\lvac e^\theta\Phi^{(\beta)*(\mu_{r})}(u_r)\cdots e^\theta\Phi^{(\beta)*(\mu_1)}(u_1)
{\phi}^{[\beta](\lambda_j)}(w_1)e^{-\theta}\cdots{\phi}^{[\beta](\lambda_{s'})}(w_{s'})e^{-\theta}\rvac\nonumber
\\
&=[u^{\mu}
w^{\lambda}]\lvac \prod_{1\le i\le r}^{\leftarrow}
e^{(r+1-i)\theta}\Phi^{(\beta)*(\mu_{i})}(u_i)
e^{-(r+1-i)\theta}
\prod_{1\le j\le s'}^{\rightarrow}
e^{(s'+1-j)\theta}{\phi}^{[\beta](\lambda_j)}(w_j)e^{-(s'+1-j)\theta}\rvac\nonumber\\
&=[u^\mu w^\lambda]
\frac{\prod\limits_{j=1}^{s'}\big(1+\beta w_j\big)^{s'+1-j}}{\prod\limits_{i=1}^{r}\big(1+\beta u_i^{-1}\big)^{r+1-i}}
\lvac \prod_{1\le i\le r}^{\leftarrow}\Phi^{(\beta)*(\mu_{i})}(u_i)
\prod_{1\le j\le s'}^{\rightarrow}
{\phi}^{[\beta](\lambda_j)}(w_j)\rvac,
\label{eq:coef_c_b}
\end{align}
where we used Lemma \ref{Lem:beta-theta}
in the last equality.
By Wick's theorem, the vacuum expectation value is zero unless
\( s' = r\). So now we set $s'=r.$ Then the vacuum expectation values in \eqref{eq:coef_c_b} is written as
\begin{align*}
\frac{d_{\lambda}(b)}{d_{\mu}(c)}[u^{\mu}w^{\lambda}]
\prod_{i=1}^{r}\frac{\prod_{k=1}^{\mu_i-1}(1-u_i c_k)}{\prod_{k=1}^{\lambda_j
}(1-w_j^{-1}b_k)}
\lvac \prod_{1\le i\le r}^{\leftarrow}\Phi^{(\beta)*(\mu_{i})}(u_i)
\prod_{1\le j\le s'}^{\rightarrow}
{\phi}^{[\beta](\lambda_j)}(w_j)\rvac.
\end{align*}
Since
the rational function in this expression is a product form such that all variables are separated,
the function \eqref{eq:coef_c_b} is, up to scaler $d_{\lambda}(b)d_{\mu}(c)^{-1}$, equal to the Pfaffian of a $2r\times 2r$ skew-symmetric matrix such that each entry
involving only a pair of variables in $u_r,\ldots,u_1,w_1,\ldots,w_r$.

Consider the factor involving $u_i$ and $u_j:$
\begin{equation}
[u_i^{\mu_i}u_j^{\mu_j}]f_i(u_i)f_j(u_j)\lvac\Phi^{(\beta)*}(u_{i})\Phi^{(\beta)*}(u_{j})\rvac
\quad\text{for $1\le i,j\le r$,}
\label{eq:uiuj}
\end{equation}
where
\begin{equation*}
f_i(u)=\frac{\prod_{k=1}^{\mu_{i}-1}(1-uc_k)}{(1+\beta u^{-1})^{r+1-i}}.
\end{equation*}
Since $f_i(u_i)$ is a series in $u_i$ of the highest degree $\mu_{i}-1$, and the vaccum expectation value in \eqref{eq:uiuj} does not contain terms that could increase the power of $u_{i} $ and $u_{j}$ by \eqref{Phibet-vacu}, thus the coefficient extraction $[u_{i}^{\mu_{i}} u_{j}^{\mu_{j}}]$ yields zero.

Hence the Pfaffian reduces to a determinant, and we have
\begin{equation}
{}_{(GP,c)}\langle\mu|
\lambda\rangle_{(gq,b)}=d_{\lambda}(b)d_{\mu}(c)^{-1}\mathrm{det}({D}),
\end{equation}
where
\begin{equation}
D_{ij}=[u_i^{\mu_i}w_j^{\lambda_j}]\frac{(1+\beta w_j)^{r+1-j}}{(1+\beta u_{i}^{-1})^{r+1-i}}
\frac{\prod_{k=1}^{\mu_i-1}(1-u_{i}c_k)}{\prod_{k=1}^{\lambda_j}(1-w_j^{-1}b_k)}
\lvac\Phi^{(\beta)*}(u_{i}){\phi}^{[\beta]}(w_{j})\rvac.
\end{equation}
We compute by using \eqref{Lem:phi-Phi}
\begin{align*}
\lvac\Phi^{(\beta)*}(u_{i}){\phi}^{[\beta]}(w_{j})\rvac&=\frac{1+\beta u_{i}^{-1}}{1+\beta w_j}
\lvac e^{\theta}\Phi^{(\beta)*}(u_{i}){\phi}^{[\beta]}(w_{j})e^{-\theta}\rvac\\
&=\frac{1+\beta u_{i}^{-1}}{1+\beta w_j}
\frac{u_iw_j}{1-u_iw_j}.
\end{align*}
Thus we have
\begin{equation}
D_{ij}=[u_i^{\mu_i}w_j^{\lambda_j}]\frac{(1+\beta w_j)^{r-j}}{(1+\beta u_{i}^{-1})^{r-i}}
\frac{\prod_{k=1}^{\mu_i-1}(1-u_{i}c_k)}{\prod_{k=1}^{\lambda_j}(1-w_j^{-1}b_k)}
\frac{u_iw_j}{1-u_iw_j}.
\end{equation}

We show the inner product $
{}_{(GP,c)}\langle\mu|
\lambda\rangle_{(gq,b)}$ vanishes unless $s=r$.
In fact, if the length $s$ of $\lambda$ is odd, that is, $s'=s+1=r$ and $\lambda_{s'}=0$. Then we have
\begin{align*}
{D}_{ir}
&=[u_i^{\mu_i}w_r^0]
\frac{\prod_{k=1}^{\mu_i-1}(1-u_{i}c_k)}{(1+\beta u_{i}^{-1})^{r-i}}
\frac{u_iw_r}{1-u_iw_r}=0
\end{align*}for $1\leq i\leq r$ because the constant term of $w_r$ in the series is zero, and hence $\det D=0.$

Now we consider the case when the length $r$ of $\mu$  is odd, so $r'=r+1$ and $\mu_{r+1}=0$. The proof proceeds similarly, but now $r +1= s'$ by Wick's theorem. We introduce a matrix $A=(A_{ij})_{1\leq i,j\leq r+1}$ by
\begin{align*}
A_{ij}&=\lvac
\Phi^{(\beta)*(\mu_{i})}(u_i)
\phi^{[\beta](\lambda_{j})}(w_j)
\rvac\quad \text{for $1\le i\le r$ and $1\le j\le r+1$},\\
A_{r+1,j}&=\lvac
\phi^{(\beta)*}(u_{r+1})\phi^{[\beta](\lambda_{j})}(w_j)
\rvac
\quad \text{for $1\leq j\leq r+1$}.
\end{align*}
Then we have  \begin{align*}
{}_{(GP,c)}\langle\mu|
\lambda\rangle_{(gq,b)}=d_{\lambda}(b)d_{\mu}(c)^{-1}
\det(\widetilde{D}),
\end{align*}
where $\widetilde{D}$ is a $(r+1)\times (r+1)$ matrix with entries
\begin{align*}
\widetilde{D}_{ij}&=[u_{i}^{\mu_{i}}w_j^{\lambda_j}]\frac{(1+\beta w_j)^{r+2-j}}{\big(1+\beta u_{i}^{-1}\big)^{r+2-i}}\frac{\prod\limits_{k=1}^{\mu_{i}-1}(1-u_{i}c_k)}{\prod\limits_{k=1}^{\lambda_{j}}(1-w_j^{-1}b_k)}{A}_{ij}.
 \end{align*}

For $1\le j \le r$, it follows from Lemma \ref{Lem:phi-Phi} that
$\widetilde{D}_{r+1,j}=0$.  If $\lambda_{s+1}=\lambda_{r'}> 0$, then in this case $\widetilde{D}_{r+1,r+1}=0$ again from Lemma \ref{Lem:phi-Phi}. This leads to $
{}_{(GP,c)}\langle\mu|
\lambda\rangle_{(gq,b)}=0$.  Thus if we want $
{}_{(GP,c)}\langle\mu|
\lambda\rangle_{(gq,b)}$ to be non-zero, it is necessary to require  $\lambda_{r+1}=\lambda_{s'}=0$. So $s$ is odd and  $r=\ell(\mu) = \ell(\lambda)=s$. Then we have $\widetilde{D}_{r+1,r+1}=1$.  Now the inner product then simplifies to
 \begin{align*}
&{}_{(GP,c)}\langle\mu|
\lambda\rangle_{(gq,b)}
=
d_{\lambda}(b)d_{\mu}(c)^{-1}
\det(D),
\end{align*}
where ${D}$ is the $r\times r$ submatrix of $\tilde{D}$ with $1\le i,j\le r$.
From Lemma \ref{lem:integral2}, one sees that the entries of $D$ is
given by the same
formula \eqref{eq:D in h} as the case of $r$ is even.
\end{proof}

\subsection{Proof of the Duality Lemma \ref{thm: gq VEV}}

\begin{lem}\label{lem:incl} We have
${}_{(GP,c)}\langle\mu|
\lambda\rangle_{(gq,b)} = 0$ unless $\lambda\subset\mu $.
\end{lem}
\begin{proof}
We know equations \eqref{eq:pair c b}, \eqref{eq:D in h} hold.
Suppose there exists an index $a$ with $1 \leq a \leq r$ such that $\lambda_a>\mu_a $. For $p = 0, \dots, r-a$ and $q = 0, \dots, a-1$, we have the inequalities:
\[
\mu_{a+p} \leq \mu_a - p < \lambda_a - p \leq \lambda_{a-q} - p - q.
\]

Now let us consider the summation index $k$ in matrix entry $D_{a+p,a-q}$ given by \eqref{eq:D in h}.
Given that  $a+p\geq a-q$,
the binomial coefficients in \eqref{eq:D in h} vanish unless $k \leq (a+p) - (a-q) = p+q$, and thus
\[
\mu_{a+p} - \lambda_{a-q} + k \leq \mu_{a+p} - \lambda_{a-q} + p + q < 0.
\]
This implies $h_{\mu_{a+p}-\lambda_{a-q}+k} = 0$ for all admissible $k$, so $D_{a+p,a-q} = 0$. These vanishing entries form a zero block in rows $a$ to $r$ and columns $1$ to $a$, rendering $D$ singular. Thus $\det(D) = 0$ and ${}_{(GP,c)}\langle\mu|
\lambda\rangle_{(gq,b)} = 0$, completing the proof.
\end{proof}

\begin{proof}[Proof of Lemma \ref{lem:dual GP vs gq}]
By using Theorem \ref{thm:b-c}, we can assume $r=s$ and then
\begin{align*}{}_{(GP,b)}\langle\mu|
\lambda\rangle_{(gq,b)}={\rm det}({D}|_{c=b}),
\end{align*}
where ${D}|_{c=b}$ is an $r\times r$ matrix  given by
\begin{align}
D_{ij}|_{c=b}&=\sum_{k=0}^\infty
\beta^k \binom{i-j}{k}
h_{\mu_i-\lambda_j+k}(b_1,\ldots,b_{\lambda_j};-b_1,\ldots,-b_{\mu_i-1})\nonumber\\
&=\begin{cases} \sum_{k=0}^\infty\beta^k\binom{i-j}{k} h_{\mu_{i}- \lambda_j+k}(b_{\mu_{i}},\cdots,b_{\lambda_j}) & \text{if $\mu_i\le \lambda_j$}\\
0&\text{if $\mu_i=\lambda_j+1$},
\\
\sum_{k=0}^\infty\beta^k\binom{i-j}{k} e_{\mu_{i}- \lambda_j+k}(-b_{\lambda_{j}+1},\cdots,-b_{\mu_i-1}),
&\text{if $\mu_i>\lambda_j+1$}.
\end{cases}
\label{eq:Dbb}
\end{align}

By Lemma \ref{lem:incl},
we can assume $\lambda_i\le \mu_i$ for $1\le i\le r.$
We will show $
D_{ii}|_{c=b}=\delta_{\lambda_i\mu_i}$ for $1\le i\le r.$
If $\mu_i=\lambda_i$ we have
$
D_{ii}|_{c=b}=
1$ from \eqref{eq:Dbb}.
Suppose $\lambda_i<\mu_i$. If $\mu_i=\lambda_i+1$ then $D_{ii}=0$
again from \eqref{eq:Dbb}. If $\mu_i>\lambda_i+1$
we have from \eqref{eq:Dbb} $$D_{ii}|_{c=b}= e_{\mu_{i}- \lambda_i}(-b_{\lambda_{i}+1},\cdots,-b_{\mu_i-1})=0
$$
because the number of
variables is $\mu_i-\lambda_i-1.$
Next we will show that $D|_{c=b}$ is upper triangular.
In fact, if $i<j$ we have
$\mu_i>\mu_j\ge \lambda_j$ so
\begin{align*}
D_{ij}|_{c=b}&=
\sum_{k=0}^\infty\beta^k\binom{i-j}{k} e_{\mu_{i}- \lambda_j+k}(-b_{\lambda_{j}+1},\cdots,-b_{\mu_i-1})=0
\end{align*}
because the number of variables is $\mu_i-\lambda_j-1.$
Therefore
we have $$
\det D|_{c=b}=
\delta_{\lambda_1\mu_1}\cdots \delta_{\lambda_r\mu_r}=\delta_{\lambda\mu},
$$
as desired.
\end{proof}

\subsection{Conjecture on factorial $gp$-functions }\label{sec:conj}
At the end of this section, we state conjectures on the factorial \(gp\)-functions.
\begin{conj}\label{conj:gp}
Let
\begin{align*}
&{\Phi}_i^{[\beta](k)}:=\prod_{\ell=1}^{k-1}(1+\beta b_\ell)\sum_{j=0}^\infty h_j(b_1\dots,b_k)\Big(\Phi_{i+j}^{[\beta]}-\frac{1}{2}(-\frac{\beta}{2})^{i+j}\Big).
\end{align*}
$$|\lambda\rangle_{gp,b}:=\Phi_{\lambda_1}^{[\beta](\lambda_1)}e^{-\theta}\Phi_{\lambda_2}^{[\beta](\lambda_2)}e^{-\theta}\cdots\Phi_{\lambda_r}^{[\beta](\lambda_r)}(\phi_0+1)\rvac.$$ Then we have   \begin{equation*}
gp_\lambda(y\,|\,b)=\lvac  e^{\mathcal{H}^{[\beta]}(y)}|\lambda\rangle_{(gp,b)}.
\end{equation*}
\end{conj}

 In order to prove this conjecture it suffices  to prove
 \begin{conj}
 For strict partitions $\lambda,\mu$, we have
     \begin{align*}
_{(GQ,b)}\langle\mu|\lambda\rangle_{p,b}=\delta_{\mu,\lambda},\quad
 _{(GQ,b)}\langle\mu|:=(|\mu\rangle_{(GQ,b)})^*.
\end{align*}
 \end{conj}
This leads to a formula of $gp_\lambda(x\,|\,b)$
as a sum of  Pfaffians similar to Corollary \ref{cor:gq-GP}.

\section{Coefficients in terms of  double Grothendieck polynomials}\label{sec_coef_gro_pol}
In this section, we give formulas expressing the general expansion coefficients in terms of Grothendieck polynomials of type A. We also discuss coincidences among these expansion coefficients (Corollary \ref{cor:coins}).

\subsection{Double Grothendieck polynomials}
For each permutation $w\in S_\infty=\bigcup_{n=1}^\infty S_n$, there is
a \emph{double Grothendieck polynomial}
$\G_w(b;c)$ due to Lascoux and Sch\"utzenberger \cite{LaSc} defined as follows.

Define the operators $\pi_i$ for $i\ge 1$ on
$\Z[\beta][b;c]:=\Z[\beta][b_1,b_2,\ldots][c_1,c_2,\dots]$ by
$$
\pi_i=\partial_i \circ (1+\beta b_{i+1}),\quad
\partial_i:=\frac{1-s_i^{(b)}}{b_i-b_{i+1}},
$$
where $s_i$ exchanges $b_i$ and $b_{i+1}.$
Let $w_0^{(n)}$ denote the longest element of $S_n.$

There is a unique family $\{\G_{w}(b;c)\mid w\in S_\infty\}
$
of polynomials such that
for $n\ge 1$ $$
\G_{w_0^{(n)}}(b;c)=\prod_{i+j\le n}
b_i\ominus c_j.$$
and
\begin{equation}\label{eq_p_i_equi}
\pi_i \G_w(b;c)=\begin{cases}
    \G_{ws_i}(b;c) & \text{if $\ell(ws_i)<w$,}\\
-\beta \G_w(b;c)    & \text{if $\ell(ws_i)>w$.}
\end{cases}
\end{equation}
We call the polynomial $\G_w(b;c)$ by double Grothendieck polynomial.
\begin{remark}
    In the existing literature $\G_{w_0^{(n)}}(b;c)$ is defined as the product of $b_i\oplus c_j$ instead of $b_i\ominus c_j$. In this setting, our definition coincide with $\G_w(b,\ominus c)$. We use this sign convention to be consistent with the factorial Grothendieck polynomial defined in \eqref{eq: def fac Groth}. Our convention is consistent with \cite{LLS}.
\end{remark}

We define $\G_{w}(b):=\G_{w}(b;0)$ for $ w\in S_\infty$.
Then $\G_{w}(b)$ are polynomials in $\Z[\beta][b_1,b_2,\dots]$ known as Grothendieck polynomial.

\begin{prop}\cite[Proposition 5.10]{LLS}\label{Prop_inv_gro} $\G_w(b;c)=\G_{w^{-1}}(\ominus c;\ominus b)$. In particular $\G_w(0;c)=\G_{w^{-1}}(\ominus c;0)=\G_{w^{-1}}(\ominus c)$.
\end{prop}


\begin{remark}\label{rem:D and T}
For $i\ge 1$ and  $f\in \Z[\beta][b]$, we have
$$
\pi_i f=\frac{1}{c(\alpha_i)}\left(s_i f-(1+\beta c(\alpha_i))f
\right),$$
where $c(\alpha_i)=b_{i+1}\ominus b_i$.
The right hand side has the same form of the operator $T_i-\beta$, where $T_i$ is defined in \eqref{eq_T_i}. Thus $T_i=\pi_i+\beta$. We use this fact to prove Theorem \ref{thm:gp expansion}.
\end{remark}

For a positive integer $r$, a permutation $w\in S_\infty$ is called $r$-\emph{Grassmannian} if it satisfies
$$w(1)<\cdots<w(r),\;
w(r+1)<w(r+2)<\cdots.$$
Note that such a permutation $w$ is determined by the values
$w(1),\dots,w(r).$
Let $\lambda=(\lambda_1,\lambda_2,\dots,\lambda_r)$ be a partition of length at most $r.$
Then there exists a unique $r$-Grassmannian permutation
$w_\lambda^{(r)}$
such that
\begin{equation}
w_\lambda^{(r)}(i)-i=\lambda_{r-i+1}\quad (1\le i\le r).\label{eq:Grassmann vs partition}
\end{equation}
Let $\lambda=(\lambda_1,\dots,\lambda_r)$ be a partition of length at most $r$. We define $$\rho_{\lambda_i}^{(r)}:=s_{r-i+\lambda_i}\cdots s_{r-i+2}s_{r-i+1}\text{ if }\lambda_i\geq1.$$ If $\lambda_i=0$ we define $\rho_{\lambda_i}^{(r)}:=e$ then
\begin{equation}\label{eq_w_lam_A_k}
 w_\lambda^{(r)}=\rho_{\lambda_r}^{(r)}\cdots \rho_{\lambda_1}^{(r)}.
\end{equation}

\begin{example}
    Let $\lambda=(3,1,0)$ be a partition of length 3. Consider the following tableau is for $r=3$.
\[\begin{ytableau}
  s_3& s_4&s_5&\cdots\\
  s_2&s_3&s_4&\cdots\\
 s_1&s_2&s_3&\cdots
\end{ytableau}\]
The 3-Grassmannian permutation corresponding to $\lambda$ is given by $$w_\lambda^{(3)}=\rho_0^{(3)}\rho_{1}^{(3)}\rho_{3}^{(3)}=(s_2)(s_5s_4s_3).$$
\end{example}

The following fact is well-known (cf. \cite{Mc}, \cite{IN13}).
\begin{prop}\label{prop_fac_gro_vs_dou_gro}
For any partition $\lambda$ of length at most $r$, we have
\begin{equation}\label{eq_detmin_formula}
\G_{w_\lambda^{(r)}}(b;c)=G_\lambda(b_1,\ldots,b_r|c)=
\frac{\det([b_i|c]^{\lambda_j+r-j}(1+\beta b_i)^{j-1})}{\prod_{1\le i<j\le r}(b_i-b_j)}.
\end{equation}
\end{prop}


In otherwords, the factorial Grothendieck polynomial $G_\lambda(b_1,\ldots,b_r|c)$  defined in \eqref{eq: def fac Groth} for any partition $\lambda$ of length at most $r$ is equal to the double Grothendieck
polynomial $\G_{w_{\lambda}^{(r)}}(b;c)$ associated with the $r$-Grassmannian permutation $w_{\lambda}^{(r)}$.

So in particular,
 if we denote the non-equivariant version of
$G_\lambda(x_1,\ldots,x_r|0)$ by
$G_\lambda(x_1,\ldots,x_r)$, we have
\begin{equation}\label{eq:non-eq-Grothendieck}
G_\lambda(x_1,\ldots,x_r)
=\G_{w_\lambda^{(r)}}(x).
\end{equation}
\begin{remark}\label{prop_w_lam}
    Let $\lambda\in \SP(r)$ be a strict partition of length $r$. We recall $w_\lambda$ from \eqref{eq_w_lam_A}.
    Define $\lambda':=\lambda-\delta_r$  where $\delta_r:=(r,r-1,\dots,2,1)$. Note that $\lambda'$ is a partition(need not be strict) of length at most $r$. Then $$w_\lambda w_{\delta_r}^{-1}=w_{\lambda'}^{(r)}.$$ This could be seen through the following tableau is for $r=3$.
\[\begin{ytableau}
*(gray) & *(gray) & *(gray)&  s_3& s_4&s_5&\cdots\\
 \none & *(gray)&*(gray)& s_2&s_3&s_4&\cdots\\
 \none & \none &*(gray)& s_1&s_2&s_3&\cdots
\end{ytableau}\]
 Moreover, Let $\lambda$ and $\mu$ be two elements in $\SP(r)$ such that $\delta_r\subseteq \lambda\subseteq \mu$. Then
    $$w_\mu w_\lambda^{-1}=w_{\mu'}^{(r)}(w_{\lambda'}^{(r)})^{-1},$$ where $\lambda'=\lambda-\delta_r$ and $\mu'=\mu-\delta_r$ and $w_{\mu'}^{(r)}$, $w_{\lambda'}^{(r)}$ are defined using \eqref{eq_w_lam_A_k}.
\end{remark}

\subsection{Expansion of factorial $gp$-function and factorial $GQ$-function}

Using the expansion in \eqref{exp_eq_GP}, and the factorization as in \eqref{eq: GP rho+la 1} we can write

\begin{equation}\label{exp_fac_gro}
G_\mu(x_1,\dots,x_r|c)=\sum_{\lambda:\lambda\subset \mu}C^{}_{\mu\lambda}(c,b)\,G_\lambda(x_1,\dots,x_r|b),
\end{equation}
where $C^{}_{\mu\lambda}(c,b)=C_{\mu\sqcup\delta_r,\lambda\sqcup\delta_r}^{GP}(c,b).$ Again, using the factorization as in \eqref{eq: GQ rho+la 1}, we have the following expansions
\begin{align}
GQ_\mu(x\,|\,c)&=\sum_{\lambda:\lambda\subset \mu,\ell(\lambda)=\ell(\mu)}C^{GQ}_{\mu\lambda}(c,b)\,GQ_\lambda(x\,|\,b),\label{ex_GQ}\\
gp_\lambda(y\,|\,b)&=\sum_{\mu:\lambda\subset\mu,\ell(\lambda)=\ell(\mu)}C_{\lambda\mu}^{gp}(b,c)\,gp_\mu(y\,|\,c)\label{ex:gp}.
\end{align}

Then  $C_{\lambda\mu}^{gp}(b,c)=C^{GQ}_{\mu\lambda}(c,b)=\langle GQ_\mu(x\,|\,c),gp_\lambda(y\,|\,b)\rangle$.


\begin{lem}\label{prop:staircase}
$$
gp_{\delta_r}(y\,|\,b)=
\sum_{\mu\in \SP(r)}
\G_{w_\mu w_{\delta_r}^{-1}}(b;c)\,
gp_{\mu}(y\,|\,c)
$$
\end{lem}
\begin{proof}
We consider $\mu \in \SP(r)$ and substitute $x=b_{\delta_r}$ in \eqref{ex_GQ}, where $b_{\delta_r}=( b_r,\dots, b_1,0,\cdots)$. Then, using the vanishing property as in Proposition \ref{prop_van_pro}
$$GQ_\mu(b_{\delta_r}|c)=C_{\mu\delta_r}^{GQ}(c,b)GQ_{\delta_r}(b_{\delta_r}|b).$$
Now using the factorization property as in \eqref{eq: GQ rho+la 1} and \eqref{eq:GQ rho 0} we have:
\begin{align*}
GQ_{\delta_r}(x_1,\ldots,x_r|b)
&=GQ_{\delta_r}(x_1,\ldots,x_r),\\
GQ_{\mu}(x_1,\ldots,x_r|c)
&=GQ_{\delta_r}(x_1,\ldots,x_r)G_{\mu'}(x_1,\ldots,x_r|c),
\end{align*}
where $\mu'=\mu\setminus\delta_r$. Therefore, from these equations, we have
\begin{align*}
C_{\mu\delta_r}^{GQ}(c,b)=\frac{GQ_\mu(b_{\delta_r}|c)}{GQ_{\delta_r}(b_{\delta_r}|b)}&=\frac{GQ_{\delta_r}(b_{\delta_r})G_{\mu'}(b_{\delta_r}|c)}{GQ_{\delta_r}(b_{\delta_r})}\\
&=G_{\mu'}(b_1,b_2,\dots,b_r|c)=\G_{w_{\mu'}^{(r)}}(b;c)=\G_{w_\mu w_{\delta_r}^{-1}}(b;c).
\end{align*}
The last equality follows from Remark \ref{prop_w_lam}, and the second last  equality from \eqref{eq_detmin_formula}.
This completes the proof.
\end{proof}

Similarly, we consider the expansions of $gq_\lambda(y\,|\,b)$ and $GP_{\mu}(x\,|\,c)$ as in \eqref{exp_eq_gq} and \eqref{exp_eq_GP} respectively and use the vanishing and factorization property of $GP_{\mu}(x\,|\,c)$ to deduce the following:

\begin{lem}\label{prop:staircase_1}
$$
gq_{\delta_r}(y\,|\,b)=
\sum_{\mu\in \SP(r)}
\G_{w_\mu w_{\delta_r}^{-1}}(b;c)\,
gq_{\mu}(y\,|\,c)
$$
\end{lem}
\begin{proof}
The proof follows similarly as Lemma \ref{prop:staircase} using the factorization in \eqref{eq: GP rho+la 1}.
\end{proof}

\begin{cor}
  $gp_{(r,r-1,\ldots,1)}(y\,|\,b)$ and  $gq_{(r,r-1,\ldots,1)}(y\,|\,b)$ are symmetric in $b_1,\ldots,b_r.$
\end{cor}

Let $\lambda\in \SP(r)$ and $\lambda\setminus \delta_r$ has $k$ many removable boxes with contents ${j_1},{j_2},\dots, {j_k}$. For example if we consider the strict partition $\lambda=(5,3,1)$ as in Example \ref{ex_st_par} then $\lambda-\delta_3$ has 2 removable boxes with content $4$ and 2. If we consider the strict partition $\lambda=(5,4,1)$ then $\lambda-\delta_3$ has only 1 removable boxes with content $3$.  Note that if $m\neq \ell$ then $|j_m-j_\ell|\geq  2$, thus $s_{j_m}s_{j_{\ell}}=s_{j_\ell}s_{j_{m}}$. Without loss of generality we can assume $j_1>j_2>\dots>j_k$. Also for any subset $I=\{i_1,i_2,\dots,i_n\}\subset \{{j_1},{j_2},\dots, {j_k}\}$ define $w_I=s_{i_1}s_{i_2}\dotsm s_{i_n}$. If $I=\emptyset$ then we consider $w_I=e$ (the identity element of $\Winf$).

\begin{thm}\label{thm:gp expansion}
Let $\lambda=(\lambda_1,\dots,\lambda_r)\in \SP(r)$ and $\lambda\setminus\delta_r$ has exactly $k$ many removable boxes with contents ${j_1},{j_2},\dots, {j_k}$. Then the coefficients $C_{\lambda\mu}^{gp}(b,c)$ and $C_{\lambda\mu}^{gq}(b,c)$ defined in \eqref{ex:gp} and \eqref{exp_eq_gq} are given by
\begin{equation}\label{eq_cof_arb_lam}
C_{\lambda\mu}^{gp}(b,c)=C_{\lambda\mu}^{gq}(b,c)=\sum_{I\subset\{j_1,\dots,j_k\}}\beta^{|I|}\G_{{w_\mu w_\lambda^{-1}} w_I}(b;{c}),
\end{equation}
\end{thm}

\begin{proof}
We address the proof for the coefficient $C_{\lambda\mu}^{gp}(b,c)$. The argument proceeds by induction on the
length of $|\lambda|-|\delta_r|$.
    For $\lambda=\delta_r$ using Lemma \ref{prop:staircase}
    \begin{equation}\label{eq_delta_r}
     gp_{\delta_r}(y\,|\,b)=\sum_{\mu\in \SP(r) }\G_{w_\mu w_{\delta_r}^{-1}}(b;c)gp_\mu(y\,|\,c).
    \end{equation}
Consider the strict partition $\xi_r$ of length $r$ by $\xi_r:=(r+1,r-1,r-2,\dots,2,1)$.  Note that $\xi_r$ is the unique strict partition of size $|\delta_r|+1$ containing $\delta_r$. Thus $$w_{\xi_r}=s_rw_{\delta_r}>w_{\delta_r}.$$
For any $\xi_r\subseteq\mu$ we have $w_\mu w_{\delta_r}^{-1}s_r=w_\mu w_{\xi_r}^{-1}<w_\mu w_{\delta_r}^{-1}$. Apply $T_r$ on both sides of \eqref{eq_delta_r} and using Theorem \ref{thm_ope_des} we have
 $$gp_{\xi_r}(y\,|\,b)=\sum_{\mu\in \SP(r)}T_r\Big(\G _{w_\mu w_{\delta_r}^{-1}}(b;c)\Big)gp_\mu(y\,|\,c).$$

If $\mu=\delta_r$ then $T_r(\G _{w_\mu w_{\delta_r}^{-1}}(b;c))=T_r(\G_{e}(b;c))=0$. For $\mu
\supset \xi_r$,  using \eqref{eq_p_i_equi} we have
$$T_r(\G _{w_\mu w_{\delta_r}^{-1}}(b;c))=(\pi_r+\beta)(\G_{ w_\mu w_{\delta_r}^{-1}}(b;c))=\G_{w_\mu w_{\xi_r}^{-1}}(b;c)+\beta\G_{w_\mu w_{\delta_r}^{-1}}(b;c).$$

Thus
 $$gp_{\xi_r}(y\,|\,b)=\sum_{\mu~:~\xi_r\subset\mu; \ell(\mu)=r}\Big(\G_{w_\mu w_{\xi_r}^{-1}}(b;c)+\beta \G_{w_\mu w_{\delta_r}^{-1}}(b;c)\Big)gp_\mu(y\,|\,c).$$
Thus \eqref{eq_cof_arb_lam} holds for $\lambda=\xi_r$.

 For arbitrary $\lambda$, let $j_k$ be the content of a removable box $\lambda\setminus\delta_r$ and
$\eta$ be obtained from $\lambda$ by removing box of content ${j_k}$. Note that  \begin{align*}
    |\lambda|=1+|\eta| \text{ and }w_\lambda=s_{j_k}w_\eta>w_\eta.
\end{align*}  We have two possibilities. Either $\eta\setminus\delta_r$ has $k-1$ many removable box with contents ${j_1},\dots,j_{k-1}$, or, $\eta\setminus\delta_r$ has $k$ many removable box with contents ${j_1},{j_2},\dots, {j_{k-1}}, {j_{k}-1}$. 
If $\eta\setminus\delta_r$ has $k-1$ many removable box with contents ${j_1},\dots,j_{k-1}$ then using the induction hypothesis
\begin{equation}\label{eq_ind_hyp}
gp_\eta(y\,|\,b)=\sum_{\substack{\mu:~\eta\subset\mu,\\ \ell(\mu)=\ell(\eta) }}\Big(\sum_{I\subset\{j_1,\dots,j_{k-1}\}}\beta^{|I|}\G_{{w_\mu w_\eta^{-1}} w_I}(b;c)\Big)gp_\mu(y\,|\,c).
\end{equation}
One sees $|j_k-j_\ell|\geq  2$ for $\ell\neq k$. Thus $s_{j_k}s_{j_{\ell}}=s_{j_\ell}s_{j_{k}}$ and $w_Is_{j_k}>w_I$ if $I\subset\{j_1,\dots,j_{k-1}\}$.

We apply $T_{j_k}$
to the both sides of \eqref{eq_ind_hyp}.
For the left hand side, $T_{j_k}(gp_{\eta}(y\,|\,b))$ is same as $gp_{\lambda}(y\,|\,b)$ using Theorem \ref{thm:gp expansion}.
 For the right hand side, the coefficient of $gp_\eta(y\,|\,c)$ vanishes because
 $T_{j_k}\G_{w_I}(b;c)=0$ for all $I \subset \{j_1,\dots,j_{k-1}\}$.
Also, for any $\mu$ such that $\lambda\subset \mu$ we have  $w_\mu w_{\lambda}^{-1}w_I=w_\mu w_{\eta}^{-1}s_{j_k}w_I=w_\mu w_{\eta}^{-1}w_Is_{j_k} <w_\mu w_{\eta}^{-1}w_I$. Therefore using \eqref{eq_p_i_equi} \begin{align*}
&T_{j_k}\Big(\sum_{I\subset\{j_1,\dots,j_{k-1}\}}\beta^{|I|}\G_{w_\mu w_\eta^{-1} w_I}(b;c)\Big)\\
&=\sum_{I\subset\{j_1,\dots,j_{k-1}\}}\Big(\beta^{|I|}\G_{w_\mu w_\lambda^{-1} w_I}(b;c)+\beta^{{|I|}+1}\G_{{w_\mu w_\eta^{-1}} w_I}(b;c)\Big)\\
&=\sum_{I\subset\{j_1,\dots,j_{k}\}}\beta^{|I|}\G_{{w_\mu w_\lambda^{-1}} w_I}(b;{c}).
\end{align*}
From 1st line to 2nd line we use the fact $T_{j_k}=\pi_{j_k}+\beta$. From 2nd line to last line we use the fact that $w_\mu w_{\lambda}^{-1}s_{j_k}w_I=w_\mu w_{\eta}^{-1}w_I$. Thus applying $T_{j_k}$ on both sides of \eqref{eq_ind_hyp}, and we have our desired result.

For the remaining case, let $\eta\setminus\delta_r$ has $k$ many removable box with contents ${j_1},{j_2},\dots, {j_{k-1}}, {j_{k}-1}$. Then there exist some extra terms $\beta^{|I|}\G_{{w_\mu w_\eta^{-1}} w_I}(b;c)$ in the right hand side of \eqref{eq_ind_hyp}, where $j_{k}-1\in I\subset \{j_1,j_2,\ldots, j_{k-1}, j_k-1\}$. Then ${w_\mu w_\eta^{-1}} w_Is_{j_k}>{w_\mu w_\eta^{-1}} w_I$ as $s_{j_k-1}s_{j_k}\neq s_{j_k}s_{j_k-1}$. Therefore, these extra terms vanish after applying $T_{j_k}$ as $T_{j_k}=\pi_{j_k}+\beta$ and $$\pi_{j_k}\Big(\G_{{w_\mu w_\eta^{-1}} w_I}(b;c)\Big)=-\beta \G_{{w_\mu w_\eta^{-1}} w_I}(b;c).$$

The argument for the coefficient $C_{\lambda \mu }^{gq}(b,c)$ proceeds in the same manner using the Lemma \ref{prop:staircase_1}.  We only need to observe that for all $i\geq 1$, the operators $T_i$ in types $\mathrm{B}_{\infty }$ and $\mathrm{C}_{\infty }$ coincide, and the proof follows accordingly.
\end{proof}

\begin{cor}\label{cor:coins}
Consider $\lambda,\mu \in \SP(r)$ such $\lambda\subset\mu$. Then $$C_{\lambda\mu}^{gp}(b,c)=C_{\lambda\mu}^{gq}(b,c)=C_{\mu\lambda}^{GP}(c,b)=C_{\mu\lambda}^{GQ}(c,b).$$
    \end{cor}


\begin{remark}
Let $\lambda, \mu\in \SP(r)$ and $\delta_r=(r,r-1,\dots,1)$.
For $\mu\subset \lambda$, using Proposition \ref{Prop_inv_gro}, we have $$C_{\lambda\mu}^{GP}(b,c)=C_{\lambda\mu}^{GQ}(b,c)=\sum_{I\subset\{j_1,\dots,j_k\}}\beta^{|I|}\G_{w_Iw_\mu w_\lambda^{-1}}(\ominus b;\ominus c),$$
where $\mu\setminus \delta_r$  has exactly $k$ many removable box with contents ${j_1},\dots, {j_k}$.
\end{remark}

\begin{cor}
If $\eta$ is a partition of length less than or equal to $r$ then
    $$G_{\eta}(x_1,\dots,x_r|b)=\sum_{\nu:~\nu\subseteq \eta}C_{\eta\nu}(b,c)G_{\nu}(x_1,\dots,x_r|c), \text{ and }$$
    $$C_{\eta\nu}(b,c)=\sum_{I\subset\{j_1,\dots,j_k\}}\beta^{|I|}\G_{w_Iw_\nu^{(r)}(w_{\eta}^{(r)})^{-1}}(\ominus b;\ominus c),$$
    where $\nu\setminus \delta_r$  has exactly $k$ many removable box with contents ${j_1},{j_2},\dots, {j_k}$.
\end{cor}
The expansion of double Grothendieck polynomials in terms of Grothendieck polynomials is know due to \cite[Proposition 5.8]{LLS}.

\begin{cor}
Let $\lambda=(\lambda_1,\dots,\lambda_r)\in \SP(r)$ and $\delta_r=(r,r-1,\dots,1)$.
Then the coefficients in \eqref{eq_exp_eq_gq} and \eqref{eq_exp_eq_GP} are given by the following: $$C_{\lambda\mu}^{gp}(b)=C_{\lambda\mu}^{gq}(b)=\sum_{I\subset\{j_1,\dots,j_k\}}\beta^{|I|}\G_{{w_\mu w_\lambda^{-1}} w_I}(b),$$
 where $\lambda\setminus \delta_r$  has exactly $k$ many removable box with contents ${j_1},{j_2},\dots, {j_k}$. Moreover, $$C_{\lambda\mu}^{GP}(b)=C_{\lambda\mu}^{GQ}(b)=\sum_{I\subset\{j_1,\dots,j_k\}}\beta^{|I|}\G_{w_Iw_\mu w_\lambda^{-1}}(\ominus b),$$
 where $\mu\setminus \delta_r$  has exactly $k$ many removable box with contents ${j_1},{j_2},\dots, {j_k}$.
\end{cor}

By comparing the coefficients in the expansion formulae in Lemma \ref{prop:staircase} and Theorem \ref{thm:gq(xb)}, we obtain a Jacobi-Trudi type formulae of the factorial Grothendieck polynomial $ G_{\eta}(b_1,\dots,b_r|c)$ as defined in \eqref{eq: def fac Groth}, for any partition $\eta$ of length less than or equal to $r$.
\begin{cor}
Let $\eta$ be a partition of length less than or equal to $r$. Then
$ G_{\eta}(b_1,\dots,b_r|c)$ is equal to
\begin{equation*}
\frac{d_{\delta_r}(b)}{d_{\eta\sqcup \delta_r}(c)}\det\Big(\sum_{m\geq0}{i-j\choose m}\beta^mh_{\eta_i-i+j+m}(b_1,\dots,b_{r+1-j};-c_1,\dots,-c_{\eta_i+r-i})\Big)_{1\leq i,j\leq r}.
\end{equation*}
 where ${d_{\delta_r}(b)}$ and ${d_{\eta\sqcup \delta_r}(c)}$ are defined in \eqref{eq_d_lam}.
\end{cor}

\appendix


\end{document}